\RequirePackage{ifthen}
\RequirePackage{ifpdf}
\newcommand{\driverOption}{}
\ifthenelse{\boolean{pdf}}{
  \renewcommand{\driverOption}{pdftex}
} { 
  \renewcommand{\driverOption}{dvips}
}

\documentclass[reqno, 11pt, letterpaper, commented, oneside, \driverOption]{amsart}

\newboolean{isCommented}
\usepackage{geometry}
\usepackage{bbm}
\usepackage[final]{graphicx}
\usepackage{upref}
\usepackage{enumitem}
\usepackage{latexsym}
\usepackage{amssymb}
\usepackage[ansinew]{inputenc}
\usepackage[T1]{fontenc}
\usepackage{mathtools}
\usepackage{comment}
\usepackage[ruled,vlined,norelsize]{algorithm2e}
\usepackage{mycommands}

\ifthenelse{\boolean{pdf}}{
  \usepackage[final]{ps4pdf}
} { 
  \usepackage[inactive]{ps4pdf}
}
\PSforPDF{\usepackage{psfrag}}

\newcommand{\hyperrefDriverOption}{}
\ifthenelse{\boolean{pdf}}{
	\renewcommand{\hyperrefDriverOption}{pdftex}
} { 
	\renewcommand{\hyperrefDriverOption}{hypertex}
}
\usepackage[\hyperrefDriverOption,
  colorlinks = false,
  pdfauthor={Petr Gregor, Torsten Mütze},
  pdftitle={Trimming and gluing Gray codes}]
  {hyperref}

\usepackage{bookmark} 

\ifthenelse{\boolean{isCommented}} {
	\newcommand{\PG}[1]{\marginpar{\parbox{4cm}{{\small {\bf PG:} #1}}}}
	\newcommand{\TM}[1]{\marginpar{\parbox{4cm}{{\small {\bf TM:} #1}}}}
} { 
	\newcommand{\PG}[1]{}
	\newcommand{\TM}[1]{}	
}

\newtheorem{theorem}{Theorem}
\newtheorem{lemma}[theorem]{Lemma}

\newtheorem{conjecture}[theorem]{Conjecture}

\theoremstyle{definition}

\theoremstyle{remark}

\newtheorem{question}[theorem]{Question}

\long\def\symbolfootnote[#1]#2{\begingroup
\def\thefootnote{\fnsymbol{footnote}}\footnote[#1]{#2}\endgroup}

\ifthenelse{\boolean{isCommented}} {
	\geometry{
	  hmargin={25mm, 50mm},
	  marginparwidth=40mm,
	  vmargin={25mm, 25mm},
	  headsep=10mm,
	  headheight=5mm,
	  footskip=10mm
	}
} { 
	\geometry{
	  hmargin={25mm, 25mm},
	  vmargin={25mm, 25mm},
	  headsep=10mm,
	  headheight=5mm,
	  footskip=10mm
	}
}

\begin{document}

\begin{center}

\renewcommand{\thefootnote}{\fnsymbol{footnote}}

\LARGE Trimming and gluing Gray codes\footnote{An extended abstract of this paper has appeared in the Proceedings of the 34th International Symposium on Theoretical Aspects of Computer Science (STACS 2017) \cite{MR3758308}.}
\vspace{2mm}

\small

\begingroup
\begin{tabular}{l@{\hspace{2em}}l@{\hspace{2em}}l}
  \Large Petr Gregor & \Large Torsten M\"{u}tze \\[2mm]
  Department of Theoretical Computer Science & Institut f\"{u}r Mathematik \\
  and Mathematical Logic & TU Berlin \\
  Charles University & 10623 Berlin, Germany \\
  11800 Praha 1, Czech Republic & \\
  {\small\tt gregor@ktiml.mff.cuni.cz} & {\small\tt muetze@math.tu-berlin.de}
\end{tabular}%
\endgroup

\vspace{5mm}

\small

\begin{minipage}{0.8\linewidth}
\textsc{Abstract.}
We consider the algorithmic problem of generating each subset of $[n]:=\{1,2,\ldots,n\}$ whose size is in some interval $[k,l]$, $0\leq k\leq l\leq n$, exactly once (cyclically) by repeatedly adding or removing a single element, or by exchanging a single element.
For $k=0$ and $l=n$ this is the classical problem of generating all $2^n$ subsets of $[n]$ by element additions/removals, and for $k=l$ this is the classical problem of generating all $\binom{n}{k}$ subsets of $[n]$ by element exchanges.
We prove the existence of such cyclic minimum-change enumerations for a large range of values $n$, $k$, and $l$, improving upon and generalizing several previous results.
For all these existential results we provide optimal algorithms to compute the corresponding Gray codes in constant $\cO(1)$ time per generated set and $\cO(n)$ space.
Rephrased in terms of graph theory, our results establish the existence of (almost) Hamilton cycles in the subgraph of the $n$-dimensional cube $Q_n$ induced by all levels $[k,l]$.
We reduce all remaining open cases to a generalized version of the middle levels conjecture, which asserts that the subgraph of $Q_{2k+1}$ induced by all levels $[k-c,k+1+c]$, $c\in\{0,1,\ldots,k\}$, has a Hamilton cycle.
We also prove an approximate version of this generalized conjecture, showing that this graph has a cycle that visits a $(1-o(1))$-fraction of all vertices.
\end{minipage}

\vspace{2mm}

\begin{minipage}{0.8\linewidth}
\end{minipage}

\vspace{2mm}

\end{center}

\vspace{3mm}

\section{Introduction}

Generating all objects in a combinatorial class such as permutations, subsets, combinations, partitions, trees, strings etc.\ is one of the oldest and most fundamental algorithmic problems, and such generation algorithms appear as core building blocks in a wide range of practical applications, see the survey \cite{MR1491049}.
In fact, half of the most recent volume \cite{MR3444818} of Donald Knuth's seminal series \emph{The Art of Computer Programming} is devoted entirely to this fundamental subject.
The ultimate goal for algorithms that efficiently generate each object of a particular combinatorial class exactly once is to generate each new object in constant time. Such optimal algorithms are sometimes called \emph{loopless algorithms}, a term coined by Ehrlich in his influential paper \cite{MR0366085}.
Note that a constant-time algorithm requires in particular that consecutively generated objects differ only in a constant amount, e.g., in a single transposition of a permutation, in adding or removing a single element from a set, or in a single tree rotation operation. These types of orderings
have become known as \emph{combinatorial Gray codes}.
Here are two fundamental examples for this kind of generation problems:
(1)~The so-called \emph{reflected Gray code} is a method to generate all $2^n$ many subsets of $[n]:=\{1,2,\ldots,n\}$ by repeatedly adding or removing a single element.
It is named after Frank Gray, a physicist and researcher at Bell Labs, and appears in his patent \cite{gray:patent}.
The reflected Gray code has many interesting properties, see \cite[Section~7.2.1.1]{MR3444818}, and there is a simple loopless algorithm to compute it \cite{MR0366085,MR0424386}.
(2)~Of similar importance in practice is the problem of generating all $\binom{n}{k}$ many $k$-element subsets of $[n]$ by repeatedly exchanging a single element.
Also for this problem, loopless algorithms are well-known \cite{MR0349274,MR0366085,MR0424386,MR782221,MR821383,MR936104,MR995888,MR1352777} (see also \cite[Section 7.2.1.3]{MR3444818}).

In this work we consider far-ranging generalizations of the classical problems~(1) and (2).
Specifically, we consider the algorithmic problem of generating all, or almost all, subsets of $[n]$ whose size is in some interval $[k,l]$, where $0\leq k\leq l\leq n$, by repeatedly adding or removing a single element, or by exchanging a single element, as further detailed later.
The classical problems~(1) and (2) can be seen as the special cases where $k=0$ and $l=n$, or where $k=l$, respectively.
The entire parameter range in between those special cases offers plenty of room for surprising discoveries and hard research problems, as Figure~\ref{fig:thms} illustrates.

In a computer a subset of $[n]$ is conveniently represented by the corresponding characteristic bitstring $x$ of length~$n$, where all the 1s of $x$ correspond to the elements contained in the set, and the 0s to the elements not contained in the set.
E.g., for $n=5$ the subset $\{1,2,5\}$ corresponds to the bitstring $11001$.
The aforementioned subset generation problems can thus be rephrased as Hamilton cycle problems in subgraphs of the \emph{cube $Q_n$}, the graph that has as vertices all bitstrings of length~$n$, with an edge between any two vertices, i.e., bitstrings, that differ in exactly one bit.
We refer to the number of 1s in a bitstring $x$ as the \emph{weight of $x$}, and we refer to the vertices of $Q_n$ with weight $k$ as the \emph{$k$-th level} of $Q_n$.
Note that there are $\binom{n}{k}$ vertices on level~$k$.
Moreover, we let $Q_{n,[k,l]}$, $0\leq k\leq l\leq n$, denote the subgraph of $Q_n$ induced by all levels $[k,l]$.
In terms of sets, the vertices of the cube $Q_n$ correspond to subsets of $[n]$, and flipping a bit along an edge corresponds to adding or removing a single element.
Continuing the previous example, moving from the vertex $11001$ to $11101$ corresponds to adding the element 3 to the set $\{1,2,5\}$, yielding the set $\{1,2,3,5\}$.
The weight of a bitstring corresponds to the size of the set, and the vertices on level~$k$ correspond to all $k$-element subsets of $[n]$.

One of the hard instances of the aforementioned general enumeration problem in $Q_{n,[k,l]}$ is when $n=2k+1$ and $l=k+1$.
The existence of a Hamilton cycle in the graph $Q_{2k+1,[k,k+1]}$ for any $k\geq 1$ is asserted by the well-known \emph{middle levels conjecture}, raised independently in the 80's by Havel~\cite{MR737021} and Buck and Wiedemann~\cite{MR737262}.
The conjecture has also been attributed to Dejter, Erd{\H{o}}s, Trotter~\cite{MR962224} and various others, and also appears in the popular books \cite{MR2034896,MR3444818,MR2858033}.
The middle levels conjecture has attracted considerable attention over the last 30 years \cite{MR1275228,MR1350586,MR1329390,MR2046083,MR962223,MR962224,MR1268348,MR2195731,MR2609124,MR2946389,MR2548541,shimada-amano}, and a positive solution, i.e., an existence proof for a Hamilton cycle in $Q_{2k+1,[k,k+1]}$ for any $k\geq 1$, has been announced only recently.

\begin{theorem}[\cite{MR3483129}]
\label{thm:mlc}
For any $k\geq 1$, the graph $Q_{2k+1,[k,k+1]}$ has a Hamilton cycle.
\end{theorem}

The following generalization of the middle levels conjecture was proposed in \cite{MR2609124}.

\begin{conjecture}[\cite{MR2609124}]
\label{conj:gen-mlc}
For any $k\geq 1$ and $c\in\{0,1,\ldots,k\}$, the graph $Q_{2k+1,[k-c,k+1+c]}$ has a Hamilton cycle.
\end{conjecture}

Conjecture~\ref{conj:gen-mlc} clearly holds for all $k\geq 1$ and $c=k$ as $Q_{2k+1,[0,2k+1]}=Q_{2k+1}$, so this is problem (1) from before.
It is known that the conjecture also holds for all $k\geq 1$ and $c=k-1$ \cite{MR1887372, locke-stong:03} and $c=k-2$ \cite{MR2609124}.
By Theorem~\ref{thm:mlc} we know that it also holds for all $k\geq 1$ and $c=0$.
As far as small cases are concerned, computer experiments show that $Q_{2k+1,[k-c,k+1+c]}$ indeed has a Hamilton cycle for all $k\leq 6$ and all $c\in\{0,1,\ldots,k\}$. The largest instance in this range not yet covered by the aforementioned general results is $Q_{13,[3,10]}$ with 8008 vertices.

Another generalization of Theorem~\ref{thm:mlc} in a slightly different direction, which still remains a special case in our general framework, is the following result.

\begin{theorem}[\cite{MR3759914}]
\label{thm:sat2}
For any $n\geq 3$ and $k\in\{1,2,\ldots,n-2\}$, the graph $Q_{n,[k,k+1]}$ has a cycle that visits all vertices in the smaller bipartite class.
\end{theorem}

The idea for the proof of Theorem~\ref{thm:sat2} based on induction over $n$ was first presented in \cite{MR737021}.
In that paper, the theorem was essentially proved conditional on the validity of the hardest case $n=2k+1$, the middle levels conjecture, which was established only much later, see Theorem~\ref{thm:mlc}.
In \cite{MR3759914}, Theorem~\ref{thm:sat2} was proved unconditionally, and the proof technique was refined further to also prove Hamiltonicity results for the so-called bipartite Kneser graphs, another generalization of the middle levels conjecture.

Conjecture~\ref{conj:gen-mlc} and Theorem~\ref{thm:sat2} immediately suggest the following common generalization:
For which intervals $[k,l]$ does the cube $Q_{n,[k,l]}$ have a Hamilton cycle?
The graph $Q_{n,[k,l]}$ is bipartite with the two partition classes given by the parity of weight of the vertices, and it is clear that a Hamilton cycle can exist only if the two partition classes have the same size, which happens only for odd dimension $n$ and between two symmetric levels $k$ and $l=n-k$, the case covered by Conjecture~\ref{conj:gen-mlc}, or for even dimension $n$ and $[k,l]=[0,n]$.
However, we may slightly relax this question, and ask for a long cycle.
To this end, we denote for any bipartite graph $G$ by $v(G)$ the number of vertices of $G$, and by $\delta(G)$ the difference between the larger and the smaller partition class.
Note that in any bipartite graph $G$ the length of any cycle is at most $v(G)-\delta(G)$, i.e., the length of a cycle that visits all vertices in the smaller partition class.
We call such a cycle a \emph{saturating cycle}, see Figure~\ref{fig:cycles}~(b).
Observe that if both partition classes have the same size, i.e., $\delta(G)=0$, then a saturating cycle is a Hamilton cycle. Hence saturating cycles naturally generalize Hamilton cycles for unbalanced bipartite graphs.
The right common generalization of Conjecture~\ref{conj:gen-mlc} and Theorem~\ref{thm:sat2} therefore is:

\begin{question}
\label{quest:sat}
For which intervals $[k,l]$ does the cube $Q_{n,[k,l]}$ have a saturating cycle?
\end{question}

\begin{figure}
\centering
\includegraphics[scale=0.916]{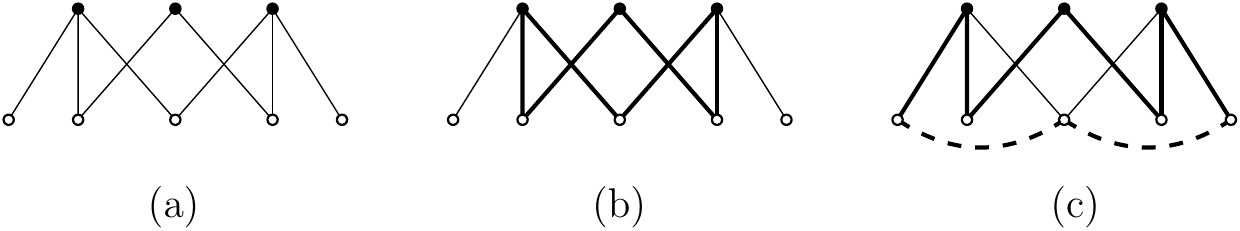} 
\caption{(a) An example of an unbalanced bipartite graph $G$ with $\delta(G)=2$.
(b) A saturating cycle of $G$.
(c) A tight enumeration of $G$, where the dashed edges represent distance-2 steps.}
\label{fig:cycles}
\end{figure}

A saturating cycle necessarily omits exactly $\delta(Q_{n,[k,l]})$ vertices from the larger bipartite class.
However, if we insist on all vertices of $Q_{n,[k,l]}$ to be included in a cycle, then this can be achieved by allowing steps where instead of only a single bitflip, two bits are flipped.
This can be viewed as augmenting the underlying graph $Q_{n,[k,l]}$ by adding distance-2 edges.
In this case we may ask for a cyclic enumeration of all vertices of $Q_{n,[k,l]}$ that minimizes the number of these `cheating' distance-2 steps, i.e., for an enumeration that has only $\delta(Q_{n,[k,l]})$ many distance-2 steps. This is clearly the smallest possible number of distance-2 steps.
We call such an enumeration a \emph{tight enumeration}, see Figure~\ref{fig:cycles}~(c).
A tight enumeration can be seen as a travelling salesman tour of length $v(Q_{n,[k,l]})+\delta(Q_{n,[k,l]})$ through all vertices of $Q_{n,[k,l]}$, where distances are measured by Hamming distance, i.e., the number of bitflips.
We may ask in full generality:

\begin{question}
\label{quest:enum}
For which intervals $[k,l]$ is there a tight enumeration of the vertices of $Q_{n,[k,l]}$?
\end{question}

Observe that if both partition classes of $Q_{n,[k,l]}$ have the same size, i.e., $\delta(Q_{n,[k,l]})=0$, then a tight enumeration is a Hamilton cycle in this graph.
Note also that Question~\ref{quest:enum} is a sweeping generalization of the following well-known result regarding problem~(2) mentioned before, first proved in \cite{MR0349274}.

\begin{theorem}[\cite{MR0349274}]
\label{thm:comb}
For any $n\geq 2$ and $1\leq k\leq n-1$ there is a cyclic enumeration of all weight~$k$ bitstrings of length~$n$ such that any two consecutive bitstrings differ in exactly 2 bits.
\end{theorem}

In fact, several of the subsequently mentioned results of this paper will be proved by extending the original approach from \cite{MR0349274} to prove Theorem~\ref{thm:comb}.

\subsection{Our results}

In this work we answer Question~\ref{quest:sat} and Question~\ref{quest:enum} for a large range of values $n$, $k$ and $l$.
The different ranges of parameters covered by our results are illustrated in Figure~\ref{fig:thms}.
Moreover, we provide optimal algorithms to compute the corresponding saturating cycles/tight enumerations.

\begin{figure}
\centering
\includegraphics[scale=0.916]{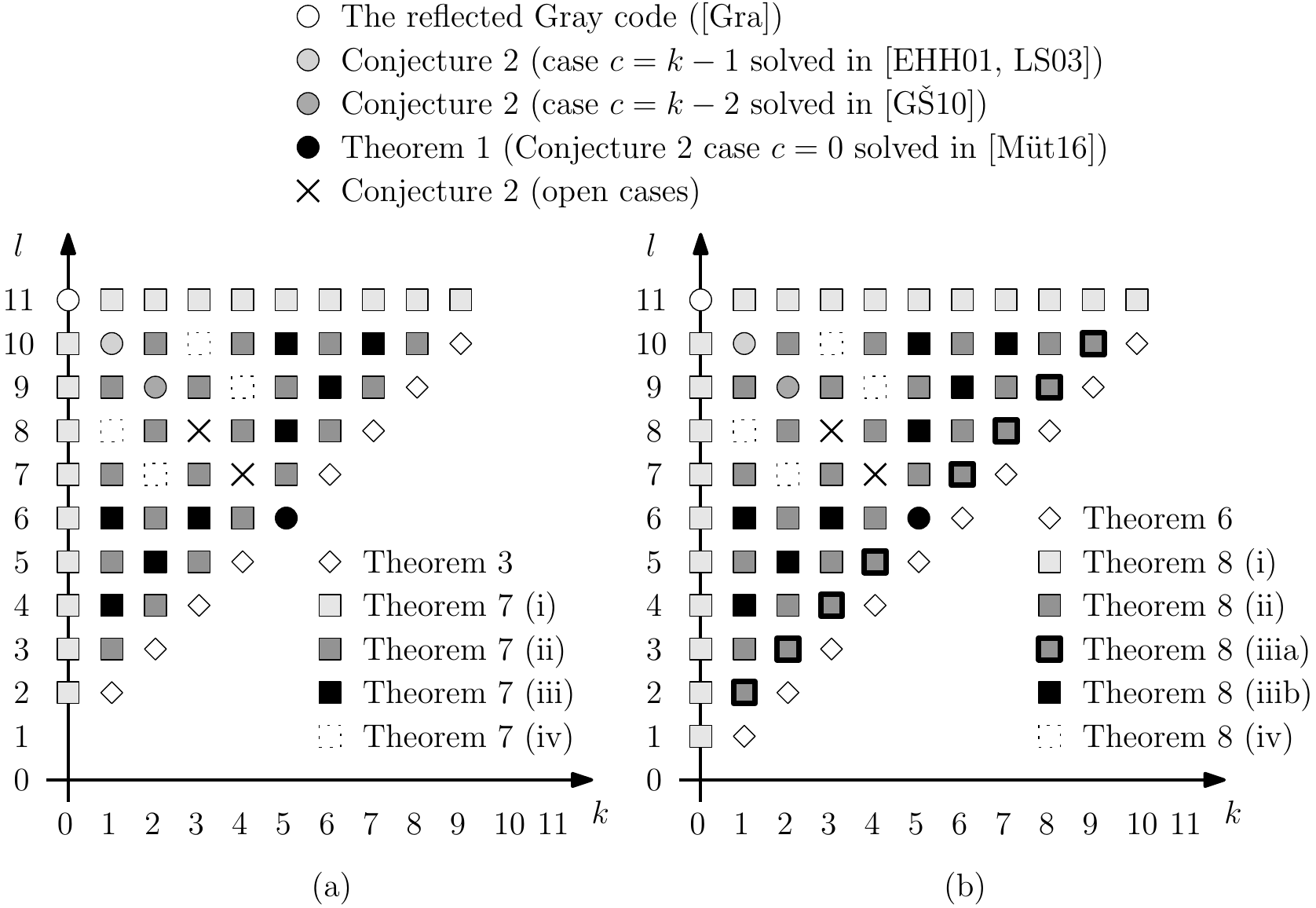} 
\caption{The different cases of $k$ and $l$ covered by (a) Theorem~\ref{thm:sat} on saturating cycles and (b) Theorem~\ref{thm:enum} on tight enumerations in $Q_{n,[k,l]}$ for the case $n=11$.
A more extensive animation of the entire parameter space $(n,k,l)$ is available on the second author's website \cite{www}.}
\label{fig:thms}
\end{figure}

Our first set of results resolves Question~\ref{quest:sat} positively for all possible values of $k$ and $l$ except the cases covered by Conjecture~\ref{conj:gen-mlc}, see Figure~\ref{fig:thms}(a). Note that the case $l=k+1$ is already covered by Theorem~\ref{thm:sat2}.

\begin{theorem}
\label{thm:sat}
For any $n\geq 3$ the graph $Q_{n,[k,l]}$ has a saturating cycle in the following cases:
\begin{enumerate}
\item[(i)]  If\/ $0=k<l\leq n$ or $0\leq k<l=n$, and $l-k\geq 2$.
\item[(ii)] If\/ $1\leq k<l\leq n-1$ and $l-k\geq 2$ is even.
\item[(iii)] If\/ $1\leq k<l\leq \lceil n/2\rceil$ or $\lfloor n/2\rfloor \leq k<l\leq n-1$, and $l-k\geq 3$ is odd.
\item[(iv)] If\/ $1\leq k<l\leq n-1$ and $l-k\geq 3$ is odd, under the additional assumption that $Q_{2m+1,[m-c,m+1+c]}$, $c:=(l-k-1)/2$, has a Hamilton cycle for all $m=c,c+1,\ldots,(\min(k+l,2n-k-l)-1)/2$.
\end{enumerate}
\end{theorem}

Our second set of results resolves Question~\ref{quest:enum} positively for all possible values of $k$ and $l$ except the cases covered by Conjecture~\ref{conj:gen-mlc}, see Figure~\ref{fig:thms}(b). Note that the case $l=k$ is already covered by Theorem~\ref{thm:comb}.

\begin{theorem}
\label{thm:enum}
For any $n\geq 3$ there is a tight enumeration of the vertices of $Q_{n,[k,l]}$ in the following cases:
\begin{enumerate}
\item[(i)]    If\/ $0=k<l\leq n$ or $0\leq k<l=n$.
\item[(ii)]   If\/ $1\leq k<l\leq n$ and $l-k\geq 2$ is even.
\item[(iiia)] If\/ $1\leq k<l\leq n-1$ and $l-k=1$.
\item[(iiib)] If\/ $1\leq k<l\leq \lceil n/2\rceil$ or $\lfloor n/2\rfloor \leq k<l\leq n-1$, and $l-k\geq 3$ is odd.
\item[(iv)]   If\/ $1\leq k<l\leq n-1$ and $l-k\geq 3$ is odd, under the additional assumption that $Q_{2m+1,[m-c,m+1+c]}$, $c:=(l-k-1)/2$, has a Hamilton cycle for all $m=c,c+1,\ldots,(\min(k+l,2n-k-l)-1)/2$.
\end{enumerate}
\end{theorem}

Note that the last part~(iv) of Theorems~\ref{thm:sat} and \ref{thm:enum} is conditional on the validity of Conjecture~\ref{conj:gen-mlc}.
In fact, recall that by the partial results on Conjecture~\ref{conj:gen-mlc}) \cite{MR1887372, locke-stong:03,MR2609124} we know that the additional assumption in (iv) is satisfied for $m\in\{c,c+1,c+2\}$, so the statement could be slightly strengthened.

The tight enumerations we construct to prove Theorem~\ref{thm:enum} have the additional property that all distance-2 steps are within a single level, and thus never between two different levels $k$ and $k+2$.
In terms of sets, these steps therefore correspond to exchanging a single element.

Given the results from the previous two theorems, we believe that the general answer to Questions~\ref{quest:sat} and \ref{quest:enum} is positive for all possible values of $n$, $k$ and $l$, and we do not know of any counterexamples.

For all the unconditional results in Theorems~\ref{thm:sat} and \ref{thm:enum} we provide corresponding optimal generation algorithms.

\begin{theorem}
\label{thm:algo}
(a) For any interval $[k,l]$ as in case~(i) or (ii) of Theorems~\ref{thm:sat} and \ref{thm:enum}, respectively, there is a corresponding loopless algorithm that generates each bitstring of a saturating cycle or a tight enumeration of the vertices of $Q_{n,[k,l]}$ in $\cO(1)$ time. \\
(b) For any interval $[k,l]$ as in case~(iii) of Theorems~\ref{thm:sat} and \ref{thm:enum}, respectively, there is a corresponding algorithm that generates each bitstring of a saturating cycle or a tight enumeration of the vertices of $Q_{n,[k,l]}$ in $\cO(1)$ time on average.
\end{theorem}

It should be noted that the algorithms for part~(a) of Theorem~\ref{thm:algo} are considerably simpler than those for part~(b).
The reason is that the underlying constructions are entirely different.
In particular, for part~(b) we repeatedly call the (average) constant-time algorithm to compute a Hamilton cycle in $Q_{2m+1,[m,m+1]}$, $m\leq \lfloor (n-1)/2\rfloor$, an algorithmic version of Theorem~\ref{thm:mlc}, presented in \cite{DBLP:conf/esa/MutzeN15, DBLP:conf/soda/MutzeN17}, and this algorithm is admittedly rather complex.
The initialization time of our algorithms is $\cO(n)$, and the required space is $\cO(n)$.

We implemented all these algorithms in C++, and we invite the reader to experiment with this code, which can be found on our website \cite{www}.

In view of these results, the only remaining, and therefore even more interesting, open case is the question whether the cube of odd dimension has a Hamilton cycle between any two symmetric levels, i.e., Conjecture~\ref{conj:gen-mlc}. These open cases are represented by crosses in Figure~\ref{fig:thms}.
Given the results from \cite{MR2609124} and \cite{MR3483129}, the next natural step towards resolving this conjecture would be to investigate whether the graphs $Q_{2k+1,[3,2k-2]}$ or $Q_{2k+1,[k-1,k+2]}$ have a Hamilton cycle for all $k\geq 1$.

In this paper we provide the following partial result towards the general Conjecture~\ref{conj:gen-mlc}:
We prove the existence of long cycles in the graph $Q_{2k+1,[k-c,k+1+c]}$, $c\in\{0,1,\ldots,k\}$.
This approximate version of the conjecture is similar in spirit to the line of work \cite{MR1275228,MR1350586,MR1329390,MR2046083} that preceded the proof of Theorem~\ref{thm:mlc}.

\begin{theorem}
\label{thm:long}
For any $k\geq 1$ and $c\in\{0,1,\ldots,k\}$, the graph $Q_{2k+1,[k-c,k+1+c]}$ has a cycle that visits at least a $(1-\epsilon)$-fraction of all vertices, where $\epsilon:=\frac{1}{2(c+1)}\min\big(1,\exp(\frac{(c+1)^2}{k-c})-1\big)$.
In particular, for any $c$ and $k\rightarrow\infty$, the cycle visits a $(1-o(1))$-fraction of all vertices.
\end{theorem}

\subsection{Related work}

In \cite{MR0366085} an algorithm is presented that generates the vertices of $Q_{n,[k,l]}$ for an arbitrary interval $[k,l]$ such that any two consecutive vertices have Hamming distance 1 or 2, where the value 2 appears only between vertices on level~$k$ and $l$, but the Hamming distance between the first and last vertex is arbitrary, possibly $n$.
The running time of that algorithm is $\cO(n)$ per generated vertex.
In addition, this paper presents a loopless algorithm, achieving $\cO(1)$ time per generated vertex, to generate all vertices in $Q_{n,[k,l]}$ level by level, using only distance-2 steps in each level.
In particular, these enumerations are not cycles in $Q_{n,[k,l]}$, and they are not tight.

In \cite{MR3193758} the authors present algorithms for enumerating all vertices of $Q_{n,[k,l]}$ for an arbitrary interval $[k,l]$ such that any two consecutive bitstrings have Levenshtein distance at most 2 and Hamming distance at most 4.
The Levenshtein distance is the minimum number of bit insertions, deletions, and bitflips necessary to transform one bitstring into another.
Again, these enumerations are not cycles in $Q_{n,[k,l]}$ and they are not tight.
However, the corresponding generation algorithms are loopless, so they are very simple and fast.
Improving on this, as a byproduct of the results mentioned in the previous section we obtain a simple loopless algorithm to generate all vertices of $Q_{n,[k,l]}$ for an arbitrary interval $[k,l]$ such that any two consecutive bitstrings have Hamming distance and Levenshtein distance at most 2.

\subsection{Outline of this paper}

In Sections~\ref{sec:sat} and \ref{sec:enum} we present the proofs of Theorems~\ref{thm:sat} and \ref{thm:enum}, respectively.
In Section~\ref{sec:algo} we provide the corresponding optimal generation algorithms, proving Theorem~\ref{thm:algo}.
Finally, in Section~\ref{sec:long} we prove Theorem~\ref{thm:long}.

\section{Saturating cycles}
\label{sec:sat}

\subsection{Trimming Gray codes and proofs of Theorem~\ref{thm:sat}~(i)+(ii)}
\label{sec:trimming}

In this section we prove cases~(i) and (ii) from Theorem~\ref{thm:sat} by showing that the standard reflected Gray code in $Q_n$ mentioned in the introduction, see \cite{gray:patent} and \cite[Section~7.2.1.1]{MR3444818}, can be `trimmed' to any number of consecutive levels of $Q_n$ so that it visits all these vertices except possibly some vertices from the first and last levels.
This technique is a generalization of the approach presented in \cite{MR0349274} to prove Theorem~\ref{thm:comb}, and it yields the following result:

\begin{theorem}
\label{thm:trim}
For any $n\geq 3$ and $k,l$ with $0\leq k<l\leq n$ and $l-k\geq 2$, the graph $Q_{n,[k,l]}$ has a cycle that visits all vertices except possibly some vertices from levels~$k$ and $l$.
\end{theorem}

Note that if $l-k$ is even, then the first and last level of $Q_{n,[k,l]}$ are from the same bipartite class, so the cycle obtained from Theorem~\ref{thm:trim} is saturating, which immediately yields Theorem~\ref{thm:sat}~(ii).

Before we provide the proof of Theorem~\ref{thm:trim}, we first introduce a few definitions and prove several auxiliary lemmas.
For any graph $G$ whose vertices are bitstrings and for any bitstring $x$, we write $G\circ x$ for the graph obtained from $G$ by concatenating every vertex with $x$.
For $b\in\{0,1\}$ and any integer $k\geq 0$ we let $b^k$ denote the bitstring that consists of exactly $k$ many $b$-bits.
For any two bitstrings $x$ and $y$, we let $d(x,y)$ denote the number of positions in which $x$ and $y$ differ (their Hamming distance).

For any sequence $\Gamma$ of not necessarily distinct vertices in a graph we let $f(\Gamma)$ and $\ell(\Gamma)$ denote the first and the last entry of $\Gamma$.
We also define $\rev(\Gamma)$ as the reversed sequence of vertices, i.e., $f(\rev(\Gamma))=\ell(\Gamma)$ and $\ell(\rev(\Gamma))=f(\Gamma)$.
Moreover, we let $s_\Gamma(x)$ and $p_\Gamma(x)$, respectively, denote the successor and predecessor of the vertex $x$ in $\Gamma$, where we define cyclically $s_\Gamma(\ell(\Gamma))=f(\Gamma)$ and $p_\Gamma(f(\Gamma))=\ell(\Gamma)$.
We may omit the subscript $\Gamma$ whenever it is clear from the context.
For two sequences $\Gamma$ and $\Gamma'$ we let $(\Gamma,\Gamma')$ denote their concatenation.
Furthermore, if $\Gamma$ is nonempty then we let $\Gamma^-$ denote the sequence obtained from $\Gamma$ by removing the last element, i.e., $\Gamma=(\Gamma^-,\ell(\Gamma))$.

The \emph{($n$-dimensional) reflected Gray code} $\Gamma_n$ is a cyclic sequence $\Gamma_n$ of all vertices of $Q_n$ defined recursively by
\begin{subequations}
\label{eq:GC}
\begin{align}
  \Gamma_1     &= (0,1) \enspace, \\
  \Gamma_{n+1} &= (\Gamma_n\circ 0,\rev(\Gamma_n)\circ 1) \enspace, \quad n\geq 1 \enspace.
\end{align}
\end{subequations}
In other words, $\Gamma_{n+1}$ is the concatenation of $\Gamma_n$ where each vertex is augmented with an additional 0-bit in the last coordinate and the reverse of $\Gamma_n$ augmented with an additional 1-bit.
See Figure~\ref{fig:G5} for an illustration.

\begin{figure}
\centering
\includegraphics[scale=0.916]{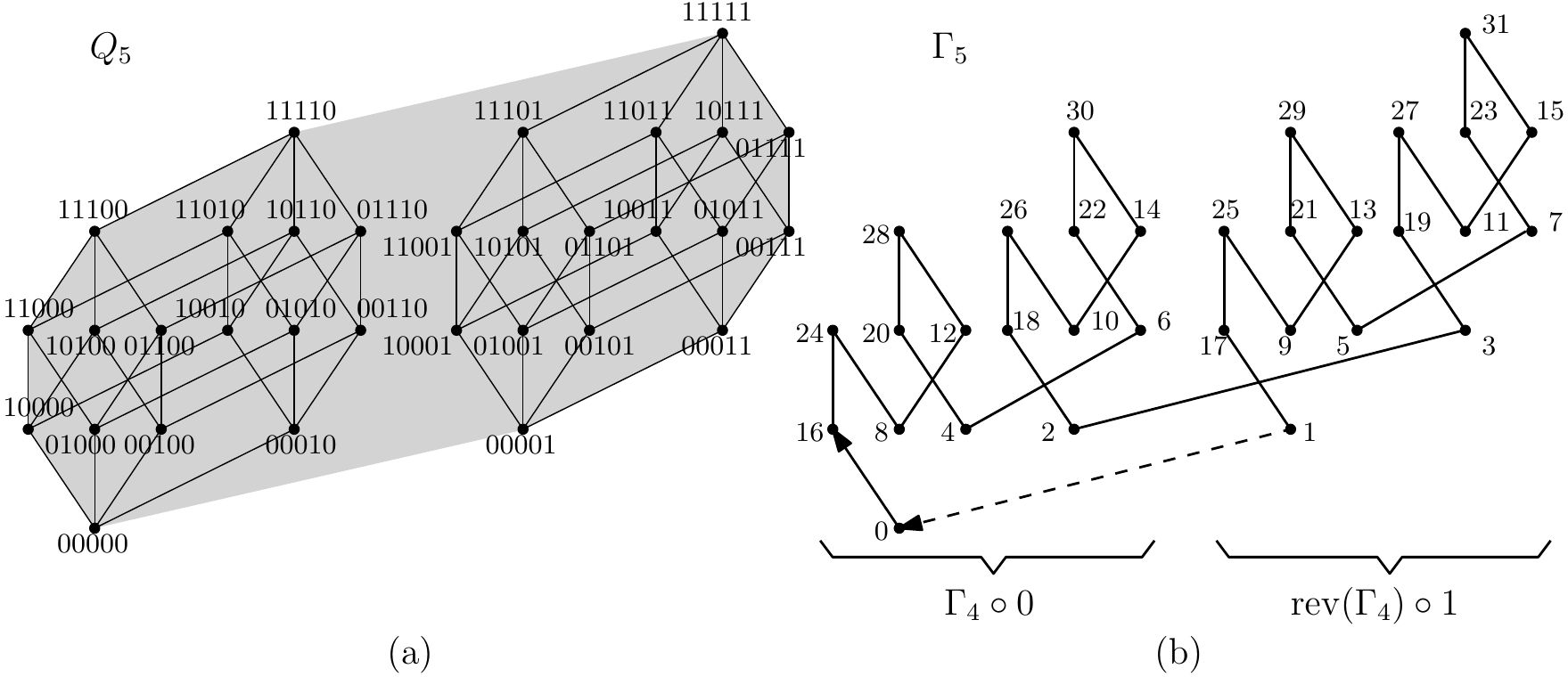} 
\caption{
(a) The hypercube $Q_5$, where the gray area represents all 16 edges along which the last bit is flipped.
(b) The reflected Gray code $\Gamma_5$ in $Q_5$, where the numbers are integer representations of the vertices in $Q_5$.
The dashed edge represents the adjacency between the last and first vertex of $\Gamma_5$.}
\label{fig:G5}
\end{figure}

The reflected Gray code $\Gamma_n$ is the standard example how to enumerate all bitstrings of length~$n$ such that any two consecutive bitstrings differ in exactly one bit.
For an explicit definition of $\Gamma_n$ and further interesting properties see \cite[Section~7.2.1.1]{MR3444818}.

For any $0\leq k \leq n$ let $\Gamma_{n,k}$ be the subsequence of $\Gamma_n$ that contains all vertices in level~$k$.
From \eqref{eq:GC} it follows that
\begin{subequations}
\label{eq:GCk}
\begin{align}
  \Gamma_{1,0}   &= (0) \enspace, \quad \Gamma_{1,1} = (1) \enspace, \\
  \Gamma_{n+1,k} &= (\Gamma_{n,k}\circ 0,\rev(\Gamma_{n,k-1})\circ 1) \enspace, \quad n\geq 1 \text{ and } 0\leq k\leq n+1 \enspace, \label{eq:GCk2}
\end{align}
\end{subequations}
where for unified treatment of border cases we define $\Gamma_{n,k}:=()$ (i.e., the empty sequence) whenever $k<0$ or $k>n$.
Furthermore, by \eqref{eq:GCk} we have
\begin{subequations}
\label{eq:GCkFL}
\begin{align}
  f(\Gamma_{n,k})    &= 1^k 0^{n-k} \enspace, \quad 0\leq k\leq n \enspace, \label{eq:GCkF} \\
  \ell(\Gamma_{n,k}) &= 1^{k-1} 0^{n-k} 1 \enspace, \quad 0<k\leq n \enspace. \label{eq:GCkL} \\
  \ell(\Gamma_{n,0}) &= f(\Gamma_{n,0}) = 0^n \enspace. \label{eq:GCkL0}
\end{align}
\end{subequations}
See Figure~\ref{fig:trimming} for an illustration.

As already observed in \cite{MR0349274}, any two consecutive vertices in $\Gamma_{n,k}$ differ in exactly two positions.
The sequence $\Gamma_{n,k}$ therefore provides an enumeration of all $k$-element subsets of $[n]$ such that any two consecutive $k$-sets differ in exchanging a single element, recall Theorem~\ref{thm:comb}.
For the purpose of self-containment we rephrase the inductive argument from \cite{MR0349274}.

\begin{lemma}[\cite{MR0349274}]
\label{lem:GCk}
For any $n\geq 2$ and $0<k<n$ let $x$ be a vertex of $\Gamma_{n,k}$ and $y:=s_{\Gamma_{n,k}}(x)$.
Then we have $d(x,y)=2$.
\end{lemma}

\begin{proof}
We prove the lemma by induction on $n$.
The statement trivially holds for $n=2$, as $\Gamma_{2,1}=(10,01)$, settling the induction basis.
For the induction step we assume that the statement holds for some $n\geq 2$ and all $0<k<n$, and we show that it also holds for $n+1$ and all $0<k<n+1$.
By \eqref{eq:GCk2} it suffices to consider only the vertices $x=\ell(\Gamma_{n,k}\circ 0)$ and $x=\ell(\rev(\Gamma_{n,k-1})\circ 1)$.
For all other choices of $x$ the claim follows easily by induction.
For the case $x=\ell(\Gamma_{n,k}\circ 0)$ we obtain
\begin{equation}
\label{eq:x1}
  x=\ell(\Gamma_{n,k})\circ 0 \eqBy{eq:GCkL} 1^{k-1} 0^{n-k} 1 0 \enspace,
\end{equation}
for all $0<k<n+1$ and
\begin{equation}
\label{eq:sx1}
  y=s_{\Gamma_{n,k}}(x) \eqBy{eq:GCk2} f(\rev(\Gamma_{n,k-1}))\circ 1=\ell(\Gamma_{n,k-1})\circ 1 \eqByM{\eqref{eq:GCkL}, \eqref{eq:GCkL0}}
  \begin{cases} 1^{k-2} 0^{n-k+1} 1 1 & \text{if } 1<k<n+1 \enspace, \\
                0^n 1 & \text{if } k=1 \enspace.
  \end{cases}
\end{equation}
Comparing the right-hand sides of \eqref{eq:x1} and \eqref{eq:sx1} shows that indeed $d(x,y)=2$.
For the case $x=\ell(\rev(\Gamma_{n,k-1})\circ 1)$ we obtain
\begin{equation}
\label{eq:x2}
  x = \ell(\rev(\Gamma_{n,k-1}))\circ 1 = f(\Gamma_{n,k-1})\circ 1 \eqBy{eq:GCkF} 1^{k-1} 0^{n-k+1} 1
\end{equation}
and
\begin{equation}
\label{eq:sx2}
  y=s_{\Gamma_{n,k}}(x) \eqBy{eq:GCk2} f(\Gamma_{n,k})\circ 0 \eqBy{eq:GCkF} 1^{k} 0^{n-k} 0
\end{equation}
for all $0<k<n+1$.
Comparing the right-hand sides of \eqref{eq:x2} and \eqref{eq:sx2} also yields $d(x,y)=2$, as desired.
This completes the proof of the lemma.
\end{proof}

Clearly, any two vertices $x,y$ in level~$k$ at distance 2 have a unique common neighbor in level~$k-1$ and a unique common neighbor in level~$k+1$, let us denote them by $\down(x,y)$ and $\up(x,y)$, respectively.
The key idea in trimming the reflected Gray code to a given sequence of consecutive levels $[k,l]$, where $l-k\geq 2$, is to replace the subpath $P$ of $\Gamma_n$ in $Q_n$ between a vertex $x$ in level~$l-1$ and its consecutive vertex $s_{\Gamma_{n,l-1}}(x)$ by the path $(x,\up(x,s(x)),s(x))$ if $P$ ascends above level~$l-1$, and between a vertex $x$ in level~$k+1$ and its consecutive vertex $s_{\Gamma_{n,k+1}}(x)$ by the path $(x,\down(x,s(x)),s(x))$ if $P$ descends below level~$k+1$.
See Figure~\ref{fig:trimming} for an illustration.

\begin{figure}
\centering
\includegraphics[scale=0.916]{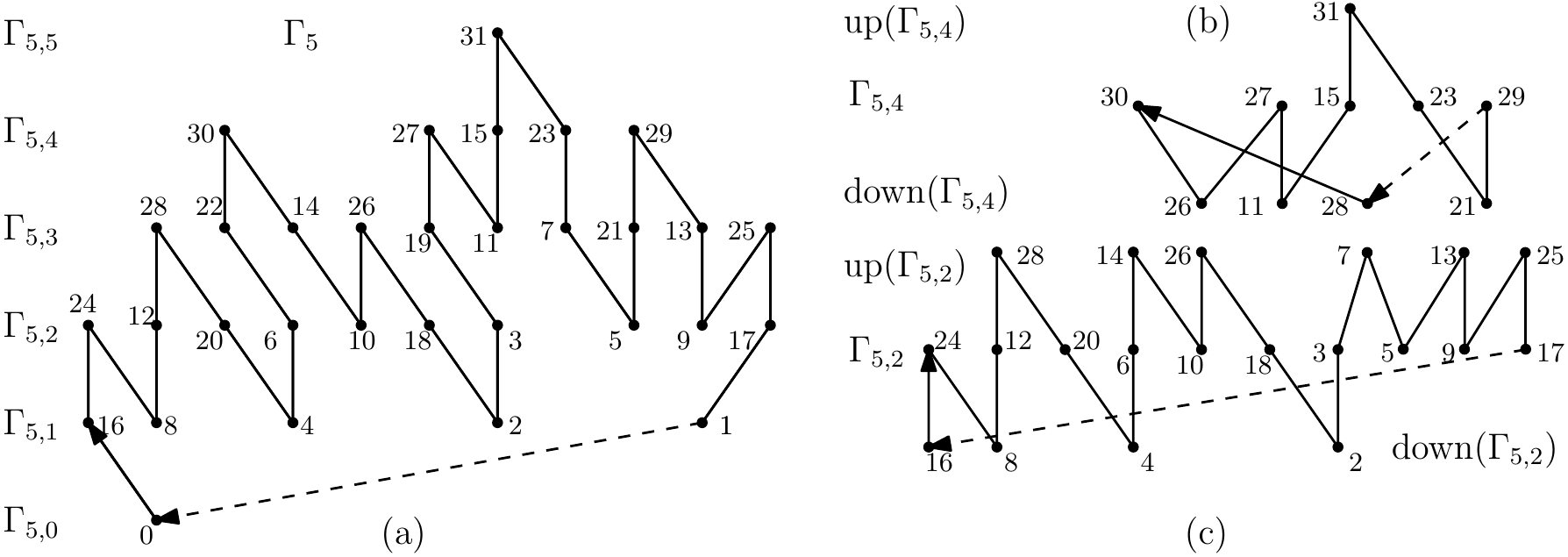} 
\caption{
(a) Schematic drawing of $\Gamma_5$ highlighting the order in which levels are visited, and the corresponding sequences $\Gamma_{5,k}$, $0\leq k\leq 5$, in each row. (b) $\Gamma_5$ trimmed to levels 3 up to 5 of $Q_5$ and the sequences $\up(\Gamma_{5,4})$, $\down(\Gamma_{5,4})$.
(c) $\Gamma_5$ trimmed to levels 1 up to 3 of $Q_5$ and the sequences $\up(\Gamma_{5,2})$, $\down(\Gamma_{5,2})$.}
\label{fig:trimming}
\end{figure}

Formally, we say that a vertex $x$ of $Q_n$ in level~$k$ is an \emph{upward vertex} if $s_{\Gamma_n}(x)$ is in level~$k+1$, and a \emph{downward vertex} if $s_{\Gamma_n}(x)$ is in level~$k-1$. Note that no other case is possible.
Thus, the reflected Gray code $\Gamma_n$ ascends via upward vertices and descends via downward vertices.
Note that $\ell(\Gamma_{n,k})$ is a downward vertex for every $0<k\leq n$ since $\Gamma_n$ starts in level~0 and ends in level~1.
For any $0<k\leq n$, we let $\up(\Gamma_{n,k})$ denote the sequence of all vertices $\up(x,s(x))$ in level~$k+1$, where $x$ is an upward vertex of $\Gamma_{n,k}$, in the order induced by $\Gamma_{n,k}$.
Similarly, for any $0\leq k<n$ we let $\down(\Gamma_{n,k})$ denote the sequence of all vertices $\down(x,s(x))$ in level~$k-1$, where $x$ is a downward vertex of $\Gamma_{n,k}$, in the order induced by $\Gamma_{n,k}$.
Note that $\up(\Gamma_{n,n-1})=(1^n)$ and $\down(\Gamma_{n,1})=(0^n)$ for every $n\geq 2$.
Moreover, we trivially have $\up(\Gamma_{n,n})=()$ and $\down(\Gamma_{n,0})=()$.
Furthermore, observe that $\up$ and $\rev$ are commutative, as are $\down$ and $\rev$ up to the last vertex.
More precisely, we have
\begin{subequations}
\begin{align}
  \up(\rev(\Gamma_{n,k}))       &= \rev(\up(\Gamma_{n,k})) \enspace, \quad 0<k\leq n \enspace, \label{eq:uprev} \\
  (\down(\rev(\Gamma_{n,k})))^- &= \rev((\down(\Gamma_{n,k}))^-) \enspace, \quad 0\leq k< n \enspace. \label{eq:downrev}
\end{align}
\end{subequations}
Moreover, we know that
\begin{equation}
\label{eq:last}
  \ell(\down(\Gamma_{n,k}))
  =\down(f(\Gamma_{n,k}),\ell(\Gamma_{n,k}))
  =\down(\ell(\rev(\Gamma_{n,k})),f(\rev(\Gamma_{n,k})))
  =\ell(\down(\rev(\Gamma_{n,k}))) \enspace.
\end{equation}

The next lemma captures the key property guaranteeing that in trimming the reflected Gray code as described above we will never visit the same vertex twice.

\begin{figure}
\centering
\includegraphics[scale=0.916]{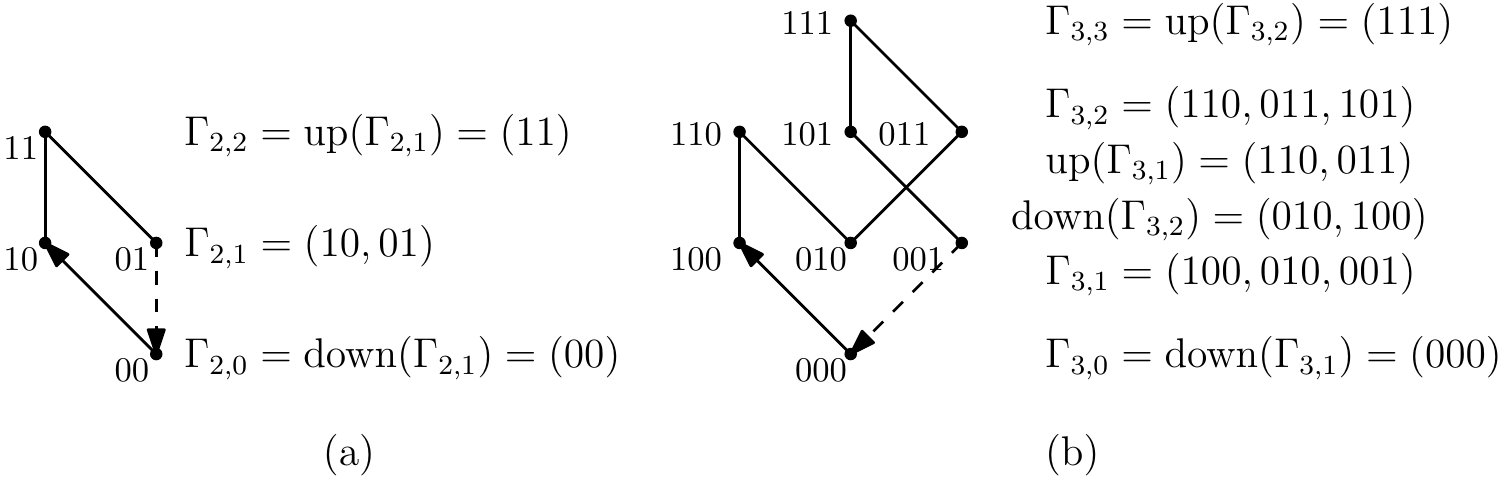} 
\caption{$\Gamma_n$, $\Gamma_{n,k}$, $\up(\Gamma_{n,k})$, and $\down(\Gamma_{n,k})$ for (a) $n=2$ and (b) $n=3$  and all possible values of $k$.}
\label{fig:G2-3}
\end{figure}

\begin{lemma}
\label{lem:GCksub}
For any $n\geq 2$ and $0<k<n$, the sequence $\up(\Gamma_{n,k})$ is a subsequence of $\Gamma_{n,k+1}$, and the sequence $(\ell(\down(\Gamma_{n,k})),(\down(\Gamma_{n,k}))^-)$ is a subsequence of $\Gamma_{n,k-1}$.
Moreover, the second sequence satisfies the following two additional conditions:
We have $\ell(\down(\Gamma_{n,k}))=f(\Gamma_{n,k-1})$, and the sequence $\down(\Gamma_{n,k})$ does not contain the vertex $\ell(\Gamma_{n,k-1})$ if $k>1$.
\end{lemma}

Note that the sequence $(\ell(\down(\Gamma_{n,k})),(\down(\Gamma_{n,k}))^-)$ referred to in Lemma~\ref{lem:GCksub} is simply $\down(\Gamma_{n,k})$ rotated to the right once.

\begin{proof}
We proceed by induction on $n$.
Figure~\ref{fig:G2-3} shows that the statement holds for $n=2$ and $n=3$, settling the induction basis.
We now assume that statement holds for some $n\geq 3$ and all $0<k<n$, and show that it also holds for $n+1$ and all $0<k<n+1$.

For $1<k<n+1$, since both $\ell(\Gamma_{n,k})$ and $\ell(\Gamma_{n+1,k})$ are downward vertices, we have
\begin{equation}
\label{eq:upseq}
  \up(\Gamma_{n+1,k}) \eqBy{eq:GCk2} (\up(\Gamma_{n,k}\circ 0),\up(\rev(\Gamma_{n,k-1})\circ 1)) \eqBy{eq:uprev} (\up(\Gamma_{n,k}\circ 0),\rev(\up(\Gamma_{n,k-1}\circ 1))) \enspace.
\end{equation}
By induction, $\up(\Gamma_{n,k})$ and $\up(\Gamma_{n,k-1})$ are subsequences of $\Gamma_{n,k+1}$ and $\Gamma_{n,k}$, respectively (for $k=n$ the sequence $\up(\Gamma_{n,k})$ is empty), so we obtain from \eqref{eq:GCk2} and \eqref{eq:upseq} that $\up(\Gamma_{n+1,k})$ is a subsequence of $\Gamma_{n+1,k+1}$.
For $k=1$ the vertex $\ell(\Gamma_{n,1})$ is a downward vertex in $\Gamma_n$, but not in $\Gamma_{n+1}$.
In fact, in $\Gamma_{n+1,1}$ there is only one more vertex after $\ell(\Gamma_{n,1})\circ 0=0^{n-1}10$, namely the vertex $f(\Gamma_n)\circ 1=0^n 1$ (recall \eqref{eq:GCk2} and \eqref{eq:GCkFL}).
In this case we therefore have
\begin{equation}
\label{eq:upseq1}
  \up(\Gamma_{n+1,1}) \eqBy{eq:GCk2} (\up(\Gamma_{n,1}\circ 0),\up(0^{n-1}10,0^n1))=(\up(\Gamma_{n,1}\circ 0),0^{n-1}11) \enspace,
\end{equation}
and since $\Gamma_{n+1,2}$ visits all vertices ending with a 0-bit before all vertices ending with a 1-bit, we conclude from \eqref{eq:upseq1} that $\up(\Gamma_{n+1,1})$ is indeed a subsequence of $\Gamma_{n+1,2}$.
This completes the proof of the first part of the lemma.

For $k=1$ we have $\down(\Gamma_{n+1,1})=(0^{n+1})$, which is indeed a subsequence of $\Gamma_{n+1,0}=(0^{n+1})$.
We now assume that $1<k<n+1$.
Here, the argument for downward vertices is more complicated than for upward vertices since the successors of $\ell(\Gamma_{n,k}\circ 0)$, $\ell(\rev(\Gamma_{n,k-1}\circ 1))$ change with the concatenation of $\Gamma_{n,k}\circ 0$ and $\rev(\Gamma_{n,k-1})\circ 1$.
By the induction hypothesis, \eqref{eq:downrev} and \eqref{eq:last} we have
\begin{subequations}
\label{eq:down-ind}
\begin{align}
  \down(\Gamma_{n,k}\circ 0)         &= \big((\down(\Gamma_{n,k}\circ 0))^-,f(\Gamma_{n,k-1}\circ 0)\big) \enspace, \label{eq:down-ind1} \\
  \down(\rev(\Gamma_{n,k-1}\circ 1)) &= \big(\rev((\down(\Gamma_{n,k-1}\circ 1))^-),f(\Gamma_{n,k-2}\circ 1)\big) \enspace. \label{eq:down-ind2}
\end{align}
\end{subequations}
From \eqref{eq:down-ind1} we obtain that the vertex $f(\Gamma_{n,k-1}\circ 0)=f(\Gamma_{n+1,k-1})$ is not contained in the sequence $(\down(\Gamma_{n,k}\circ 0))^-$ and from \eqref{eq:down-ind2} that the vertex $f(\Gamma_{n,k-2}\circ 1)=\ell(\Gamma_{n+1,k-1})$ is not contained in the sequence $\rev((\down(\Gamma_{n,k-1}\circ 1))^-)$ (recall \eqref{eq:GCkF} and \eqref{eq:GCkL}).

We now compute the two vertices in $\down(\Gamma_{n+1,k})$ added between the boundaries of $\down(\Gamma_{n,k}\circ 0)$ and $\down(\rev(\Gamma_{n,k-1}\circ 1))$ in the induction step \eqref{eq:GCk2}.
These two vertices are
\begin{align}
  \down(\ell(\Gamma_{n,k}\circ 0),f(\rev(\Gamma_{n,k-1}\circ 1))) &= \down(\ell(\Gamma_{n,k}\circ 0),\ell(\Gamma_{n,k-1}\circ 1)) \notag \\
  &\eqBy{eq:GCkL} \down(1^{k-1}0^{n-k}10, 1^{k-2}0^{n-k+1}11) \notag \\
  &= 1^{k-2}0^{n-k+1}10 \eqBy{eq:GCkL} \ell(\Gamma_{n,k-1}\circ 0) \enspace, \label{eq:down-lfrev} \\
  \down(\ell(\rev(\Gamma_{n,k-1}\circ 1)),f(\Gamma_{n,k}\circ 0)) &=\down(f(\Gamma_{n,k-1}\circ 1),f(\Gamma_{n,k}\circ 0)) \notag \\
  &\eqBy{eq:GCkF} \down(1^{k-1}0^{n-k+1}1, 1^k0^{n-k}0) \notag \\
  &= 1^{k-1}0^{n-k+2} \eqBy{eq:GCkF} f(\Gamma_{n,k-1}\circ 0) \enspace. \label{eq:down-lrevf}
\end{align}

Combining our previous observations, we obtain
\begin{align}
\down(\Gamma_{n+1,k})
 &\eqBy{eq:GCk2} \down\big((\Gamma_{n,k}\circ 0,\rev(\Gamma_{n,k-1}\circ 1))\big) \notag \\
 &\hspace{-20mm} = \Big((\down(\Gamma_{n,k}\circ 0))^-, \down\big(\ell(\Gamma_{n,k}\circ 0),f(\rev(\Gamma_{n,k-1}\circ 1))\big), (\down(\rev(\Gamma_{n,k-1}\circ 1)))^-, \notag \\
 &\down\big(\ell(\rev(\Gamma_{n,k-1}\circ 1)),f(\Gamma_{n,k}\circ 0)\big)\Big) \notag \\
 &\hspace{-20mm} \eqByM{\eqref{eq:downrev},\eqref{eq:down-lfrev},\eqref{eq:down-lrevf}}
   \big((\down(\Gamma_{n,k}\circ 0))^-, \ell(\Gamma_{n,k-1}\circ 0), \rev((\down(\Gamma_{n,k-1}\circ 1))^-), f(\Gamma_{n,k-1}\circ 0)\big) \enspace. \label{eq:downGamma}
\end{align}

As we observed before, the vertex $f(\Gamma_{n,k-1}\circ 0)$ is not contained in the sequence $(\down(\Gamma_{n,k}\circ 0))^-$, and the vertex $\ell(\Gamma_{n,k-1}\circ 0)$ is not contained in this sequence by induction, implying that \eqref{eq:downGamma} is a sequence of distinct vertices in level~$k-1$ of $Q_{n+1}$.
Moreover, as $(\down(\Gamma_{n,k}\circ 0))^-$ and $(\down(\Gamma_{n,k-1}\circ 1))^-$ are subsequences of $\Gamma_{n,k-1}\circ 0$ and $\Gamma_{n,k-2}\circ 1$, respectively, by induction, we obtain from \eqref{eq:GCk2} and \eqref{eq:downGamma} that $(\ell(\down(\Gamma_{n+1,k})),(\down(\Gamma_{n+1,k}))^-)$ is a subsequence of $\Gamma_{n+1,k-1}$, as claimed.
Furthermore, also from \eqref{eq:downGamma} we know that $\ell(\down(\Gamma_{n+1,k}))=f(\Gamma_{n,k-1}\circ 0)=f(\Gamma_{n+1,k-1})$ (recall \eqref{eq:GCkF}), as desired.
Lastly, as observed before, the vertex $\ell(\Gamma_{n+1,k-1})$ is not contained in the sequence $\rev((\down(\Gamma_{n,k-1}\circ 1))^-)$ and its last bit is 1, so it is not contained in the sequence $\down(\Gamma_{n+1,k})$ either.

This completes the proof of the lemma.
\end{proof}

We now strengthen the first part of the statement of Lemma~\ref{lem:GCksub} for upward vertices even further, which will be needed later.

\begin{lemma}
\label{lem:GCkup}
For any $n\geq 2$ and $0<k<n$ let $x$ be an upward vertex of $\Gamma_{n,k}$ and $y:=s_{\Gamma_{n,k}}(x)$.
Then $\up(x,y)$ equals $s_{\Gamma_n}(x)$ or $p_{\Gamma_n}(y)$.
\end{lemma}

In other words, Lemma~\ref{lem:GCkup} asserts that the vertex $\up(x,y)$ equals the successor of $x$ or the predecessor of $y$ on $\Gamma_n$.
Note that this implies in particular that $\up(\Gamma_{n,k})$ is a subsequence of $\Gamma_{n,k+1}$.
We shall see later that the statement of Lemma~\ref{lem:GCkup} can be strengthened even further, as demonstrated by \eqref{eq:upnext}; specifically, whether $\up(x,y)$ equals $s_{\Gamma_n}(x)$ or $p_{\Gamma_n}(y)$ is determined only by the parity of $k$.

\begin{proof}
Clearly, for any sequence $\Gamma$ of vertices and any vertex $x$ in $\Gamma$ we have
\begin{equation}
\label{eq:sprev}
  s_\Gamma(x)=p_{\rev(\Gamma)}(x) \enspace, \quad p_\Gamma(x)=s_{\rev(\Gamma)}(x) \enspace.
\end{equation}

We prove the lemma by induction on $n$.
Figure~\ref{fig:G2-3} shows that the statement holds for $n=2$ and $n=3$, settling the induction basis.
For the induction step we assume that the statement holds for some $n\geq 3$ and all $0<k<n$, and we show that it also holds for $n+1$ and all $0<k<n+1$.

First we consider an upward vertex $x$ in $\Gamma_{n+1,k}$ for $1<k<n+1$.
Recall that $\ell(\Gamma_{n,k})$ is a downward vertex, so by \eqref{eq:GCk2} $x$ must be different from $\ell(\Gamma_{n,k})\circ 0$.
Also, $x$ must be different from $\ell(\Gamma_{n+1,k})$ (which by \eqref{eq:GCk2} equals $f(\Gamma_{n,k-1})\circ 1$).
Using \eqref{eq:sprev} the claim therefore follows by induction from \eqref{eq:GCk2}.
Now consider an upward vertex $x$ in $\Gamma_{n+1,1}$.
The only case here where the previous argument fails is that $x$ might be equal to $\ell(\Gamma_{n,1})\circ 0=0^{n-1}10$ (recall \eqref{eq:GCkL}).
This is because even though $\ell(\Gamma_{n,1})=0^{n-1}1$ is a downward vertex in $\Gamma_{n,1}$, since it is the last vertex of $\Gamma_n$ it will be become an upward vertex in $\Gamma_{n+1,1}$, i.e., in $\Gamma_{n+1}$ the vertex $x=0^{n-1}10$ will be followed by $z:=0^{n-1}11$, and there is only one more entry of $\Gamma_{n+1,1}$ after $x$, namely $y=f(\Gamma_{n,0})\circ 1=0^n1$ (recall \eqref{eq:GCk2} and \eqref{eq:GCkF}).
And indeed, the vertices $x$, $y$ and $z$ satisfy the relation $s_{\Gamma_{n+1}}(x)=z=\up(x,y)$.
\end{proof}

We are now ready to prove Theorem~\ref{thm:trim} and use this result to prove Theorem~\ref{thm:sat}~(i)+(ii).

\begin{proof}[Proof of Theorem~\ref{thm:trim}]
We build the desired cycle by trimming the reflected Gray code $\Gamma_n$ to levels $[k,l]$.
Subpaths of $\Gamma_n$ within the levels $[k+1,l-1]$ remain unchanged, including the orientation.
Each subpath $P$ of $\Gamma_n$ that descends from some downward vertex $x$ at level~$k+1$ to lower levels returns back to level~$k+1$ at the vertex $y:=s_{\Gamma_{n,k+1}}(x)$.
Since $d(x,y)=2$ by Lemma~\ref{lem:GCk}, we may replace $P$ by the path $P'=(x,\down(x,y),y)$.
Note that $P'$ has the same end vertices and orientation as $P$, and visits only a single vertex at level~$k$.
After trimming all these descending paths, we visit on level~$k$ precisely the vertices of $\down(\Gamma_{n,k+1})$.
Since $\down(\Gamma_{n,k+1})$, after one rotation to the right, is a subsequence of $\Gamma_{n,k}$ by Lemma~\ref{lem:GCksub}, all these vertices are distinct, and hence distinct trimmed paths may have at most end vertices in level~$k$ in common.

Similar arguments apply for trimming subpaths of $\Gamma_n$ ascending from upward vertices at level~$l-1$ to levels above.
In this case the trimmed subpaths visit at level~$l$ precisely the vertices of $\up(\Gamma_{n,l-1})$, and since this is a subsequence of $\Gamma_{n,l}$ by Lemma~\ref{lem:GCksub}, all these vertices are distinct.
Therefore trimming correctly produces a cycle visiting all vertices in levels $[k,l]$ except the vertices from level~$k$ that are not in $\down(\Gamma_{n,k+1})$ and the vertices from level~$l$ that are not in $\up(\Gamma_{n,l-1})$.
\end{proof}

With Theorem~\ref{thm:trim} in hand, the proof of the first two cases of Theorem~\ref{thm:sat} is straightforward.

\begin{proof}[Proof of Theorem~\ref{thm:sat}~(i)]
We only consider the case $0=k<l\leq n$, the other case follows by symmetry, using that $Q_{n,[k,l]}$ is isomorphic to $Q_{n,[n-l,n-k]}$.
The proof proceeds very similarly to the proof of Theorem~\ref{thm:trim}, but we only trim the subpaths of $\Gamma_n$ ascending above level~$l-1$, so that the highest level where vertices are visited is level~$l$. Thus no trimming is applied at the bottom level~0.
We therefore obtain a cycle that visits all vertices in levels $[0,l]$, except the vertices from level~$l$ that are not in  $\up(\Gamma_{n,l-1})$.
As the cycle omits only vertices from level~$l$, it must be saturating.
\end{proof}

\begin{proof}[Proof of Theorem~\ref{thm:sat}~(ii)]
Follows immediately from Theorem~\ref{thm:trim}, using that if $l-k$ is even, then the first and last level of $Q_{n,[k,l]}$ are from the same bipartite class.
\end{proof}

\subsection{Gluing saturating cycles and proof of Theorem~\ref{thm:sat}~(iii)+(iv)}
\label{sec:gluing}

In this section we prove cases~(iii) and (iv) from Theorem~\ref{thm:sat}.
Trimming the reflected Gray code $\Gamma_n$ to levels $[k,l]$ as described in the last section does not yield a saturating cycle when $l-k\geq 3$ is odd unless $k=0$ or $l=n$.
In general the trimmed cycle omits some vertices from both levels $k$ and $l$, which are from different bipartite classes for odd $l-k$.
We therefore use a different strategy to prove Theorem~\ref{thm:sat}~(iii): We glue together several saturating cycles obtained from Theorem~\ref{thm:sat2}.
However, this gluing approach yields a saturating cycle only if all involved levels are either below or above the middle.
This is reflected by the conditions $l\leq \lceil n/2\rceil$ or $\lfloor n/2\rfloor \leq k$ in Theorem~\ref{thm:sat}~(iii).
Otherwise the omitted vertices would be from different bipartite classes, so the resulting cycle would not be saturating.
To prove Theorem~\ref{thm:sat}~(iv), we inductively glue together pairs of saturating cycles of smaller dimension.
Both proofs are similar to the approach presented in \cite{MR3759914}.

Recall that for any bipartite graph $G$, $v(G)$ denotes the number of vertices of $G$ and $\delta(G)$ denotes the difference between the sizes of the larger and the smaller partition class.
We easily compute
\begin{subequations}
\label{eq:v-delta}
\begin{align}
  v(Q_{n,[k,l]}) &= \sum_{i=k}^l\binom{n}{i} \enspace, \label{eq:vQ} \\
  \delta(Q_{n,[k,l]}) &= \left|\sum_{i=k}^l (-1)^i\binom{n}{i}\right|
  = \begin{cases}
      \binom{n-1}{k-1} + \binom{n-1}{l}            & \text{if $l-k$ is even} \enspace , \\
      \left|\binom{n-1}{k-1}-\binom{n-1}{l}\right| & \text{if $l-k$ is odd} \enspace,
    \end{cases} \label{eq:deltaQ}
\end{align}
\end{subequations}
where the degenerate cases $k=0$ and $l=n$ have to be treated by defining $\binom{n-1}{-1}=\binom{n-1}{n}=0$.
Clearly, $Q_{n,[k,l]}$ can have a Hamilton cycle only if $\delta(Q_{n,[k,l]})=0$, and from \eqref{eq:deltaQ} we conclude that this condition is satisfied if and only if $n$ is odd and $l=n-k$, or if $n$ is even and $[k,l]=[0,n]$.
Recall that Conjecture~\ref{conj:gen-mlc} asserts that this necessary condition is in fact sufficient for finding a Hamilton cycle.

\begin{proof}[Proof of Theorem~\ref{thm:sat}~(iii)]
We only consider the case $1\leq k<l\leq \lceil n/2\rceil$, the other case follows by symmetry.
By Theorem~\ref{thm:sat2}, there is a saturating cycle $C_i$ between levels $i$ and $i+1$ for every $i=k,k+2,\ldots,l-1$.
Note that each $C_i$ visits all vertices in level~$i$ since $l\leq \lceil n/2 \rceil$.
We now show how to join these $(l-k+1)/2\geq 2$ cycles into a single saturating cycle in $Q_{n,[k,l]}$.
For the reader's convenience, the approach is illustrated in Figure~\ref{fig:gluing}.

\begin{figure}
\centering
\includegraphics[scale=0.916]{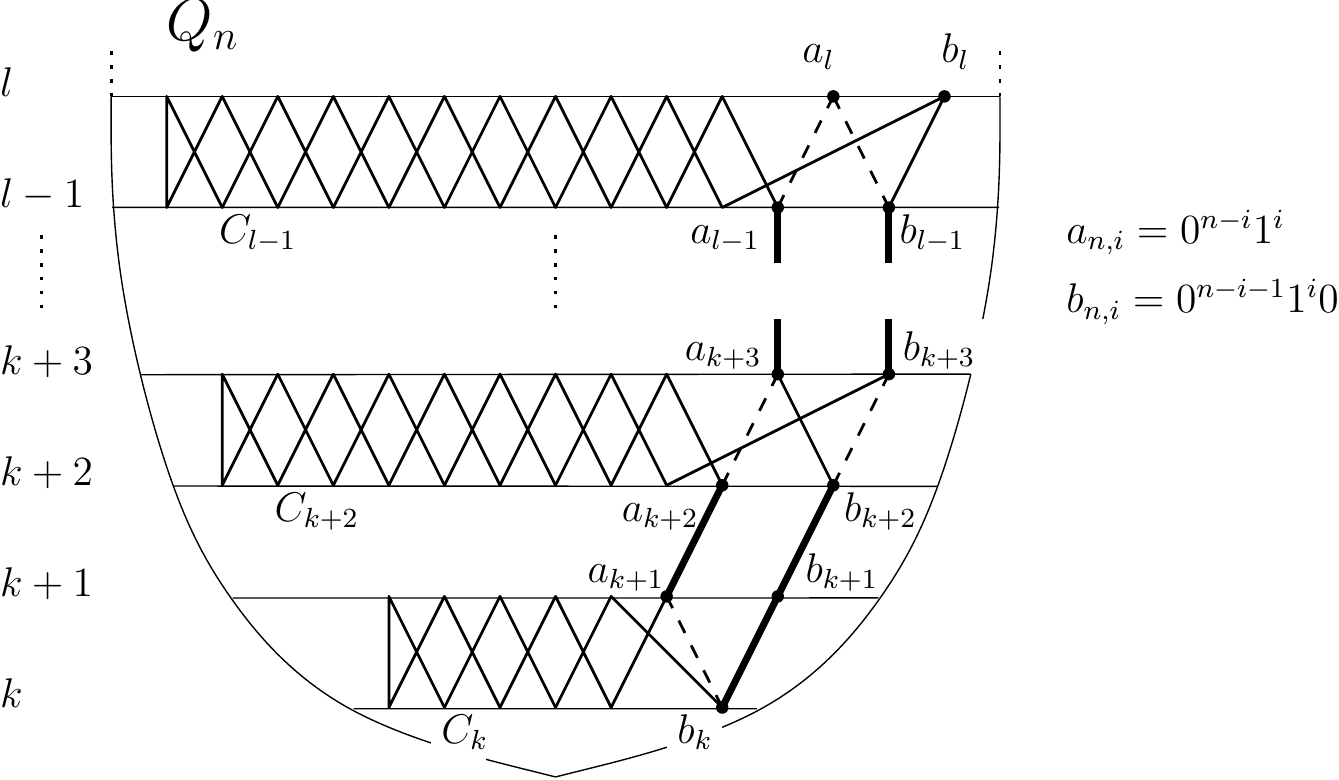} 
\caption{Notations used in the proof of Theorem~\ref{thm:sat}~(iii).
The removed edges are dashed, the added edges are bold.}
\label{fig:gluing}
\end{figure}

For any level~$i$ of $Q_n$, we define two special vertices in this level:
\begin{equation}
\label{eq:ai-bi}
  a_{n,i}:=0^{n-i}1^i \enspace, \quad b_{n,i}:=0^{n-i-1}1^i 0 \enspace.
\end{equation}
We omit the subscript $n$ whenever it is clear from the context and simply write $a_i$ or $b_i$ in this case.
In the lowest cycle $C_k$ between levels $k$ and $k+1$ we permute coordinates so that the vertices $b_k$ and $a_{k+1}$ are visited consecutively, and the vertex $b_{k+1}$ is not visited at all.
This is possible since permutation of coordinates is an automorphism of $Q_{n,[k,k+1]}$ and since there is a vertex in level~$k+1$ that is not visited by $C_k$ and any neighbor of this vertex in level~$k$ is visited by $C_k$.
Furthermore, in each of the other cycles $C_i$, $i=k+2,k+4\ldots,l-1$, we permute coordinates, independently for each cycle, so that the vertices $a_i$, $a_{i+1}$, $b_i$ and $b_{i+1}$ are visited consecutively.
This is possible since $C_i$ contains a subpath of length~3 starting at level~$i$ and any path of length~3 between two consecutive levels has the property that the coordinate changes are in different directions along the three edges.
We now modify the permuted cycles as follows:
In the lowest cycle $C_k$ we remove the edge $(b_k,a_{k+1})$ and replace it by the edge $(b_k,b_{k+1})$, thus including the previously omitted vertex $b_{k+1}$ in level~$k+1$.
In the intermediate cycles $C_i$, $i=k+2,k+4\ldots,l-3$, we remove the edges $(a_i,a_{i+1})$ and $(b_i,b_{i+1})$.
In the uppermost cycle $C_{l-1}$, we remove the edges $(a_{l-1},a_l)$ and $(b_{l-1},a_l)$, creating a new omitted vertex $a_l$ in level~$l$.
Finally, by adding the edges $(a_{i-1},a_i)$ and $(b_{i-1},b_i)$ for $i=k+2,k+4,\ldots,l-1$ we join the resulting paths to a single cycle $C$.
The cycle $C$ is saturating, as the missed vertices all lie in levels $k+1,k+3,\ldots,l$ by the condition $l\leq \lceil n/2\rceil$; i.e., they all belong to the same bipartite class.
\end{proof}

It remains to prove the last part~(iv) of Theorem~\ref{thm:sat}.
Before doing so we introduce another definition.
We say that a cycle $C$ in $Q_{n,[k,l]}$ contains a \emph{virtual 2-path} $(u,v,w)$, if $(u,v,w)$ is a 2-path of $Q_{n,[k,l]}$ where $u,w$ are in level~$l$ and $v$ is in level~$l-1$, the edge $(u,v)$ is in $C$, but the vertex $w$ is omitted by $C$, see Figure~\ref{fig:gluing-dim}.

\begin{proof}[Proof of Theorem~\ref{thm:sat}~(iv)]
We only consider the case $k\leq n-l$.
The case $k>n-l$ follows by symmetry.
We inductively prove the following strengthening of the result:
We additionally require that the only omitted vertices are in level~$l$, and if there are such vertices that the cycle contains a virtual 2-path.

For any fixed odd value of $l-k\geq 3$ we prove this statement by induction on $n$.
The induction basis for $n=l-k$, i.e., $k=0$ and $l=n$, is settled by the reflected Gray code $\Gamma_n$.

Observe that for any $n\geq l-k$ and $k=0$ this statement follows from Theorem~\ref{thm:sat}~(i).
As this cycle is obtained by trimming $\Gamma_n$ to levels $[0,l]$, the only omitted vertices are in level~$l$, and it contains a virtual 2-path:
Indeed, any omitted vertex in level~$l$ originates from replacing a path that leads from an upward vertex $x$ in level~$l-1$ to $s_{\Gamma_n}(x)$ in level~$l$, and via some additional vertices in levels $[l,n]$ back to $p_{\Gamma_n}(y)$ in level~$l$, where $y:=s_{\Gamma_{n,l-1}}(x)$, and then to the vertex $y$ in level~$l-1$, by the shorter path $(x,\up(x,y),y)$.
By Lemma~\ref{lem:GCkup}, this makes either $s_{\Gamma_n}(x)$ or $p_{\Gamma_n}(y)$ an omitted vertex in level~$l$, and creates a virtual 2-path in both cases:
If $\up(x,y)=s_{\Gamma_n}(x)$, then $p_{\Gamma_n}(y)$ is omitted and $(s_{\Gamma_n}(x),y,p_{\Gamma_n}(y))$ is a virtual 2-path.
Otherwise $\up(x,y)=p_{\Gamma_n}(y)$ and $s_{\Gamma_n}(x)$ is omitted and $(p_{\Gamma_n}(y),x,s_{\Gamma_n}(x))$ is a virtual 2-path.

For any odd $n\geq l-k$ and $k=n-l$ the statement holds by the additional assumption (that there is a Hamilton cycle in $Q_{2m+1,[m-c,m+1+c]}$ for all $m=c,c+1,\ldots,(n-1)/2$), and there are no omitted vertices at all.

We now assume that the statement holds for some $n\geq l-k$ and all $k$ in the range $0 \le k \le n-(l-k)$ where $(l-k)$ is fixed, and show that it also holds for $n+1$ and $1\leq k<(n+1)-l$. For the reader's convenience, the approach is illustrated in Figure~\ref{fig:gluing-dim}.

\begin{figure}
\centering
\includegraphics[scale=0.916]{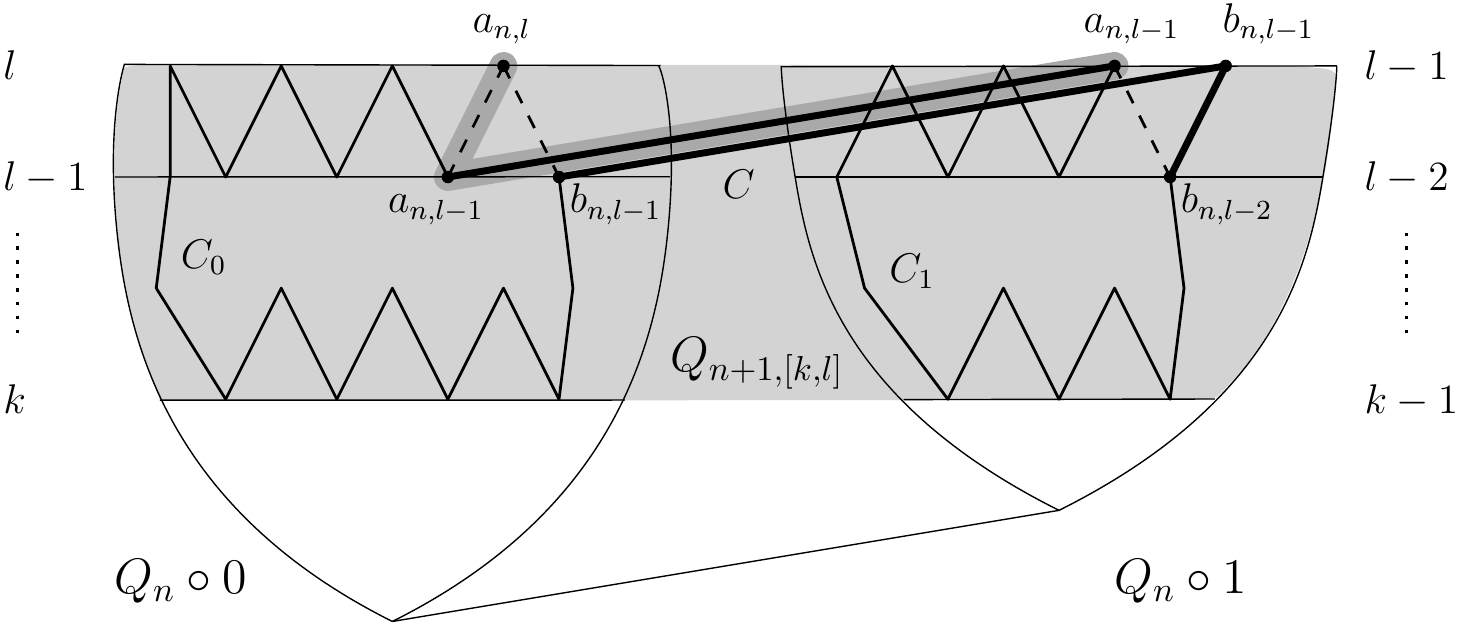} 
\caption{Notations used in the proof of Theorem~\ref{thm:sat}~(iv).
The removed edges are dashed, the added edges are bold.
The virtual 2-path for the resulting cycle $C$ in $Q_{n+1,[k,l]}$ is highlighted in dark gray.}
\label{fig:gluing-dim}
\end{figure}

By induction we know that there is a saturating cycle $C_0$ in $Q_{n,[k,l]}$. Note that $k\leq n-l$; this cycle might actually be a Hamilton cycle.
We also know that there is a saturating cycle $C_1$ in $Q_{n,[k-1,l-1]}$ satisfying all the conditions of the theorem.
In particular, by \eqref{eq:deltaQ} and since $k-1<n-(l-1)$, there will be an omitted vertex $x$ in level~$l-1$ and a virtual 2-path.
For any level~$i$ of $Q_n$, consider the special vertices $a_{n,i}$ and $b_{n,i}$ defined in \eqref{eq:ai-bi}.
We permute coordinates in $C_0$ so that the vertices $a_{n,l-1}$, $a_{n,l}$ and $b_{n,l-1}$ are visited consecutively.
Moreover, we permute coordinates in $C_1$ so that the virtual 2-path is mapped to $(a_{n,l-1},b_{n,l-2},b_{n,l-1})$.
From the cycle $C_0$ we remove the edges $(a_{n,l-1},a_{n,l})$ and $(b_{n,l-1},a_{n,l})$, creating a new omitted vertex $a_{n,l}$ in level~$l$ of $Q_n$, and we attach a 0-bit to all vertices.
In the cycle $C_1$ we remove the edge $(b_{n,l-2},a_{n,l-1})$ and replace it by the edge $(b_{n,l-2},b_{n,l-1})$, thus including the previously omitted vertex $b_{n,l-1}$ in level~$l-1$ of $Q_n$, and we attach a 1-bit to all vertices.
We connect the resulting paths by adding the edges $(a_{n,l-1}\circ 0,a_{n,l-1}\circ 1)$ and $(b_{n,l-1}\circ 0,b_{n,l-1}\circ 1)$, creating a cycle $C$ in $Q_{n+1,[k,l]}$.
The only omitted vertices of $C$ are in level~$l$ of $Q_{n+1}$, so the cycle is saturating.
Moreover, $(a_{n,l-1}\circ 1,a_{n,l-1}\circ 0,a_{n,l}\circ 0)$ is a virtual 2-path in $C$.
This completes the proof.
\end{proof}

\section{Tight enumerations}
\label{sec:enum}

In this section we present the proofs for all five cases in Theorem~\ref{thm:enum}.
For this we introduce some definitions.

We call a sequence $C$ that contains each vertex of $Q_{n,[k,l]}$ exactly once an \emph{enumeration of the vertices of $Q_{n,[k,l]}$}.
Recall that $s_C(\ell(C))=f(C)$.
The \emph{total distance} of the enumeration $C$ is $\td(C):=\sum_{u \in C} d(u,s_C(u))$.
As any two consecutive bitstrings in any enumeration have distance at least 1 and distance at least 2 if they are from the same bipartite class of $Q_{n,[k,l]}$, we have
\begin{equation}
\label{eq:tdlb}
  \td(C)\geq v(Q_{n,[k,l]})+\delta(Q_{n,[k,l]}) \enspace,
\end{equation}
where $v(Q_{n,[k,l]})$ and $\delta(Q_{n,[k,l]}$ are defined in \eqref{eq:v-delta}.
A \emph{tight enumeration} is an enumeration for which the lower bound \eqref{eq:tdlb} is attained.
Clearly, an enumeration is tight if, and only if, it has only distance-1 and distance-2 steps, and all the distance-2 steps are within the same partition class of $Q_{n,[k,l]}$. This will be the larger partition class, and there will be exactly $\delta(Q_{n,[k,l]})$ such distance-2 steps in this class.

\subsection{Proof of Theorem~\ref{thm:enum}}

For $k=l$ we have $v(Q_{n,[k,k]})=\delta(Q_{n,[k,k]})=\binom{n}{k}$ and a tight enumeration of all weight~$k$ bitstrings of length~$n$ is given by $\Gamma_{n,k}$, by Lemma~\ref{lem:GCk}. Note that this is exactly the proof of Theorem~\ref{thm:comb} presented in \cite{MR0349274}.

We now proceed to prove cases (i)--(iv) of Theorem~\ref{thm:enum}.
The proofs are very similar to the corresponding proofs for cases (i)--(iv) of Theorem~\ref{thm:sat}, and the key ideas are the same.
There are however various technical subtleties to overcome.

\begin{proof}[Proof of Theorem~\ref{thm:enum}~(ii)]
We trim the reflected Gray code $\Gamma_n$ to levels $[k,l]$, but in a slightly different fashion from the proof of Theorem~\ref{thm:trim}.
Specifically, subpaths of $\Gamma_n$ within the levels $[k,l]$ remain unchanged, including the orientation.
Moreover, each subpath $P$ of $\Gamma_n$ that descends from a downward vertex $x$ in level~$k$ returns back to level~$k$ at the vertex $y:=s_{\Gamma_{n,k}}(x)$, and we can replace $P$ by the distance-2 step $(x,y)$ by Lemma~\ref{lem:GCk}.
Similarly, each subpath $P$ of $\Gamma_n$ that ascends from an upward vertex $x$ in level~$l$ returns back to level~$l$ at the vertex $y:=s_{\Gamma_{n,l}}(x)$, and we replace $P$ by the distance-2 step $(x,y)$.
This yields an enumeration $C$ of all vertices of $Q_{n,[k,l]}$.
Moreover, since $l-k$ is even, levels $k$ and $l$ of $Q_{n,[k,l]}$ belong to the same bipartite class, so all distance-2 steps of $C$ are within the same bipartite class, and the remaining steps are distance-1 steps.
It follows that $C$ is a tight enumeration.
\end{proof}

\begin{proof}[Proof of Theorem~\ref{thm:enum}~(i)]
We only consider the case $0=k<l\leq n$, the other case follows by symmetry.
The proof proceeds very similarly as the proof of part~(ii), but we only trim the subpaths of $\Gamma_n$ ascending above level~$l$.
Thus no trimming is applied at the bottom level~0.
This yields an enumeration $C$ of all vertices of $Q_{n,[0,l]}$.
Moreover, all distance-2 steps of $C$ are within the same bipartite class in level~$l$, and the remaining steps are distance-1 steps, implying that $C$ is a tight enumeration.
\end{proof}

Theorem~\ref{thm:enum}~(iiia) is an immediate consequence of our next result, Theorem~\ref{thm:enum2}.
This theorem is the analogue of Theorem~\ref{thm:sat2} for tight enumerations.
To state the result we need to introduce some notation:
We say that an enumeration $C$ of the vertices of $Q_{n,[k,k+1]}$ contains a \emph{3-path} if $C$ contains a contiguous subsequence $(u,v,w,x)$ of vertices that form a path of length~3 in $Q_{n,[k,k+1]}$ (in this order) where $u,w$ are in level~$k$, and $v,x$ are in level~$k+1$.
Furthermore, we say that $C$ contains a \emph{switched 2-path} if $C$ contains a contiguous subsequence $(u,v,w)$ of vertices such that $(w,u,v)$ forms a path of length~2 in $Q_{n,[k,k+1]}$ where $u$ is in level~$k$ and $v,w$ are in level~$k+1$.

\begin{theorem}
\label{thm:enum2}
For any $n\geq 3$ and $0\leq k\leq n-1$ there is a tight enumeration of the vertices in $Q_{n,[k,k+1]}$.
Moreover, it contains a 3-path if $1\leq k\leq n-2$, and a switched 2-path if $k<\lfloor n/2\rfloor$.
\end{theorem}

To prove Theorem~\ref{thm:enum2}, we use again an inductive approach very similar to that from \cite{MR3759914}.

\begin{proof}
By symmetry, we may assume that $k\leq \lfloor (n-1)/2 \rfloor$.
We prove the statement by induction on $n$.

For any $n\geq 3$ and $k=0$ such a tight enumeration trivially exists, see Theorem~\ref{thm:enum}~(i).
Moreover, as the vertex $0^n$ is adjacent to all vertices in level~$1$, this enumeration must contain a switched 2-path.
Note also that the statement holds for all odd $n\geq 3$ and $k=\lfloor n/2\rfloor$ by Theorem~\ref{thm:mlc}.
Indeed, since the corresponding enumeration is a Hamilton cycle in $Q_{n,[k,k+1]}$, it must contain a 3-path.

These observations settle in particular the base cases $n=3$, $k\in\{0,1\}$.

For the induction step we assume that the statement holds for some $n\geq 3$ and all values of $k$ in the range $0\le k\leq \lfloor (n-1)/2 \rfloor$, and we prove that it also holds for $n+1$ and all $1\leq k<\lfloor n/2\rfloor$.
In any level~$i$ of $Q_n$, we consider two special vertices $a_{n,i}$ and $b_{n,i}$ as defined in \eqref{eq:ai-bi}.
By induction, there is a tight enumeration $C_0$ of the vertices of $Q_{n,[k,k+1]}$ that contains a 3-path, as $k\geq 1$.
We also know that there is a tight enumeration $C_1$ of the vertices of $Q_{n,[k-1,k]}$ that contains a switched 2-path, as $k-1<\lfloor n/2\rfloor$.
In $C_0$ we permute coordinates so that the 3-path is mapped to $(a_{n,k},a_{n,k+1},b_{n,k},b_{n,k+1})$.
Moreover, in $C_1$ we permute coordinates so that the switched 2-path is mapped to $(b_{n,k-1},b_{n,k},a_{n,k})$.
This is possible because $(b_{n,k},b_{n,k-1},a_{n,k})$ is a path of length~2 in $Q_n$.

To make the description of the following modifications easier, let us think about $C_0$ and $C_1$ as cycles in $Q_{n,[k,k+1]}$ and $Q_{n,[k-1,k]}$ that use some additional distance-2 edges, not actually present in $Q_n$, between vertices in the upper levels $k+1$ or $k$, respectively.
From $C_0$ we remove the distance-1 edge $(b_{n,k},a_{n,k+1})$ and attach a 0-bit to all vertices.
From $C_1$ we remove the distance-2 edge $(a_{n,k},b_{n,k})$ and attach a 1-bit to all vertices.
We connect the resulting paths by adding the distance-2 edge $(a_{n,k+1}\circ 0,a_{n,k}\circ 1)$ and the distance-1 edge $(b_{n,k}\circ 0,b_{n,k}\circ 1)$.
In this way, we clearly obtain an enumeration $C$ of all vertices of $Q_{n+1,[k,k+1]}$.
Moreover, all distance-2 steps of $C$ are within the same bipartite class in level~$k+1$, and the remaining steps are distance-1 steps, implying that $C$ is a tight enumeration.
Furthermore, $(b_{n,k-1}\circ 1,b_{n,k}\circ 1,b_{n,k}\circ 0,b_{n,k+1}\circ 0)$ is a 3-path in $C$, and $(a_{n,k}\circ 0,a_{n,k+1}\circ 0,a_{n,k}\circ 1)$ is a switched 2-path.
This completes the proof.
\end{proof}

\begin{proof}[Proof of Theorem~\ref{thm:enum}~(iiia)]
Follows immediately from Theorem~\ref{thm:enum2}.
\end{proof}

The proof of Theorem~\ref{thm:enum}~(iiib) is analogous to the proof of Theorem~\ref{thm:sat}~(iii), and it proceeds by gluing together several cycles obtained from Theorem~\ref{thm:enum2}.
We include the proof for the sake of completeness, as there are some minor technical differences.
In particular, to make the proof work we will need the 3-path and the switched 2-path guaranteed by Theorem~\ref{thm:enum2}.

\begin{proof}[Proof of Theorem~\ref{thm:enum}~(iiib)]
We only consider the case $1\leq k<l\leq \lceil n/2\rceil$, the other case follows by symmetry.
By Theorem~\ref{thm:enum2}, there is a tight enumeration $C_i$ of all vertices in $Q_{n,[i,i+1]}$ for every $i=k,k+2,\ldots,l-1$.
The only distance-2 steps of each $C_i$ are in level~$i+1$ since $l\leq \lceil n/2 \rceil$.
It is again useful to think of $C_i$ as a cycle in $Q_{n,[i,i+1]}$ that uses some additional distance-2 edges, not actually present in $Q_n$, between vertices in the upper level~$i+1$.
We now show how to join these $(l-k+1)/2\geq 2$ cycles to a single cycle that enumerates all vertices of $Q_{n,[k,l]}$.
In any level~$i$ of $Q_n$, we consider two special vertices $a_{n,i}=a_i$ and $b_{n,i}=b_i$ as defined in \eqref{eq:ai-bi}.

In the lowest cycle $C_k$ between levels $k$ and $k+1$ we permute coordinates so that the switched 2-path is mapped to $(b_k,b_{k+1},a_{k+1})$.
Furthermore, in each of the other cycles $C_i$, $i=k+2,k+4\ldots,l-1$, we permute coordinates, independently for each cycle, so that the 3-path is mapped to $(a_i,a_{i+1},b_i,b_{i+1})$.
We now modify the permuted cycles as follows:
In the lowest cycle $C_k$ we remove the distance-2 edge $(b_{k+1},a_{k+1})$.
In the intermediate cycles $C_i$, $i=k+2,k+4\ldots,l-3$, we remove the distance-1 edges $(a_i,a_{i+1})$ and $(b_i,b_{i+1})$.
In the uppermost cycle $C_{l-1}$, we remove the distance-1 edges $(a_{l-1},a_l)$ and $(b_{l-1},b_l)$, and we add the distance-2 edge $(a_l,b_l)$.
Finally, we connect the resulting paths by adding the distance-1 edges $(a_{i-1},a_i)$ and $(b_{i-1},b_i)$ for $i=k+2,k+4,\ldots,l-1$.
In this way, we obtain an enumeration $C$ of all vertices of $Q_{n,[k,l]}$.
Moreover, all distance-2 steps of $C$ are within the same bipartite class in levels $k+1,k+3,\ldots,l$, and the remaining steps are distance-1 steps, implying that $C$ is a tight enumeration.
\end{proof}

\begin{proof}[Proof of Theorem~\ref{thm:enum}~(iv)]
This proof works analogously to the proof of Theorem~\ref{thm:sat}~(iv) by induction over the dimension of the cube.
As in the proof of Theorem~\ref{thm:enum2}, we join pairs of `cycles' along a 3-path and a switched 2-path.
We leave the details to the reader.
\end{proof}

\section{Efficient algorithms}
\label{sec:algo}

In the following we first state some known facts about the reflected Gray code $\Gamma_n$, which will allow us to formulate a loopless algorithm to generate trimmed Gray codes, proving part~(a) of Theorem~\ref{thm:algo}; that is, algorithms for cases~(i) and (ii) of Theorems~\ref{thm:sat} and \ref{thm:enum}.
We finally describe an efficient algorithm for glued Gray codes, proving part~(b) of Theorem~\ref{thm:algo}; that is, algorithms for case~(iii) of Theorems~\ref{thm:sat} and \ref{thm:enum}.

\subsection{Explicit description of trimmed Gray codes}
\label{sec:expl}

Let $x$ be a vertex of $Q_n$ in level~$k$.
It is well-known \cite[Section~7.2.1.1]{MR3444818} that the successor of $x$ in $\Gamma_n$ is given by
\begin{equation}
\label{eq:GCsuc}
  s_{\Gamma_n}(x) = \begin{cases} x\oplus e_1     & \text{if $k$ is even} \enspace, \\
                                  x\oplus e_{i+1} & \text{if $k$ is odd} \enspace,
                    \end{cases}
\end{equation}
where $i\geq 1$ is the smallest integer with $x_i=1$, except for $x=\ell(\Gamma_n)=0^{n-1}1$ where we set $i:=n-1$, and $e_j:=0^{j-1}10^{n-j}$ for all $j=1,2,\ldots,n$.

From \eqref{eq:GCsuc} we immediately obtain the following characterization of upward vertices, and consequently also of downward vertices, in $\Gamma_n$.

\begin{lemma}
\label{lem:GCupw}
The vertex $x=(x_1,x_2,\ldots,x_n)$ in level~$k$ of $Q_n$ is an upward vertex in $\Gamma_n$ if and only if $k$ is even and $x_1=0$ or $k$ is odd and $x_{i+1}=0$, where $i$ is as defined in \eqref{eq:GCsuc}.
\end{lemma}

For any $0\leq k \leq n$ we let $u_{n,k}$ and $d_{n,k}$, respectively, denote the number of upward or downward vertices of $\Gamma_n$ in level~$k$ of $Q_n$.
Recall that by Lemma~\ref{lem:GCksub} all vertices in the sequences $\up(\Gamma_{n,k})$ and $\down(\Gamma_{n,k})$ are distinct, and therefore the length of these sequences is given by $u_{n,k}$ and $d_{n,k}$, respectively.
From Lemma~\ref{lem:GCupw} we obtain that
\begin{equation}
\label{eq:udnk}
  u_{n,k}=\binom{n-1}{k} \quad \text{and} \quad d_{n,k}=\binom{n}{k}-\binom{n-1}{k}=\binom{n-1}{k-1} \enspace.
\end{equation}

Next, we state explicit descriptions for the successor of a vertex in $\Gamma_{n,k}$ given by Tang and Liu, see \cite[Theorem 4]{MR0349274}.
Note that we use a slightly different notation with the most significant bit at the last position.

\begin{lemma}[\cite{MR0349274}]
\label{lem:upexp}
Let $n\geq 1$ and $0<k<n$ and let $x$ be an upward vertex in level~$k$ of $Q_n$.
Then we have $s_{\Gamma_{n,k}}(x)=x \oplus e_{p_1} \oplus e_{p_2}$ and $\up(x,s(x))=x \oplus e_{p_1}$ where
\begin{equation*}
  p_1=\begin{cases} i-1 & \text{if $k$ is even} \\
                    i+1 & \text{if $k$ is odd}
      \end{cases} \enspace, \quad p_2=i \enspace,
\end{equation*}
and $i\geq 1$ is the smallest integer with $x_i=1$.
\end{lemma}

Combining \eqref{eq:GCsuc} and Lemma~\ref{lem:upexp} yields that for any upward vertex $x$ in level~$k$, $0<k<n$, and for $y:=s_{\Gamma_{n,k}}(x)$ we have
\begin{equation}
\label{eq:upnext}
  \up(x,y) = \begin{cases} p_{\Gamma_n}(y) & \text{if $k$ is even} \enspace, \\
                           s_{\Gamma_n}(x) & \text{if $k$ is odd} \enspace,
             \end{cases}
\end{equation}
a further strengthening of Lemma~\ref{lem:GCkup} and of the first part of Lemma~\ref{lem:GCksub}.

For downward vertices, the description of successors in $\Gamma_{n,k}$, given by the next lemma, is more complicated.
The reason is that successors of some downward vertices change when we construct $\Gamma_{n+1,k}$ from $\Gamma_{n,k}$ and $\Gamma_{n,k-1}$ as described by \eqref{eq:GCk2}, recall the proof of Lemma~\ref{lem:GCksub}.
For the same reason, unfortunately an analogous property as \eqref{eq:upnext} does not hold for downward vertices.

\begin{lemma}[\cite{MR0349274}]
\label{lem:downexp}
Let $n\geq 1$ and $0<k<n$ and let $x$ be a downward vertex in level~$k$ of $Q_n$.
Then we have $s_{\Gamma_{n,k}}(x)=x \oplus e_{p_1} \oplus e_{p_2}$ and $\down(x,s(x))=x \oplus e_{p_1}$ where
\begin{alignat*}{3}
  &p_1=i-2 &&\text{ and } p_2=i   \quad\quad &&\text{if } k\not\equiv i \bmod 2 \enspace, \\
  &p_1=n   &&\text{ and } p_2=i   \quad\quad &&\text{if } k\equiv i \bmod 2 \text{ and } j=n \enspace, \\
  &p_1=i-1 &&\text{ and } p_2=j+1 \quad\quad &&\text{if } k\equiv i \bmod 2, j<n \text{ and } x_{j+1}=0 \enspace, \\
  &p_1=j+1 &&\text{ and } p_2=i   \quad\quad &&\text{if } k\equiv i \bmod 2, j<n \text{ and } x_{j+1}=1 \enspace,
\end{alignat*}
$i\geq 1$ is the smallest integer with $x_i=0$, and $j>i$ is the smallest integer with $x_j=1$.
\end{lemma}

Using \eqref{eq:GCsuc} and Lemmas~\ref{lem:GCupw}--\ref{lem:downexp} we are now ready to derive a loopless algorithm for generating trimmed Gray codes.

\subsection{A loopless algorithm for trimmed Gray codes}

Loopless algorithms both for the reflected Gray code $\Gamma_n$ and for its restriction $\Gamma_{n,k}$ to one level of the cube (i.e, an algorithmic version of Theorem~\ref{thm:comb}) were already provided in \cite{MR0366085,MR0424386}.
However, these two algorithms cannot simply be merged into a loopless algorithm producing the trimmed Gray code.
Instead, we provide a loopless algorithm that is based on the explicit description of successors given in the previous section.

\begin{algorithm}
\renewcommand\theAlgoLine{T\arabic{AlgoLine}}
\LinesNumbered
\DontPrintSemicolon
\SetEndCharOfAlgoLine{}
\SetNlSty{}{}{}
\SetArgSty{}
\caption{$\TrimGC(n,k,l)$}
\label{alg:trim}
\KwIn{Integers $n\geq 2$ and $0\leq k<l\leq n$, $l-k\geq 2$.}
\KwResult{The algorithm visits all vertices of the trimmed Gray code in $Q_{n,[k,l]}$ (which is a saturating cycle in cases~(i) and (ii) of Theorem~\ref{thm:sat}).}
\vspace{.2em}
$c:=k+1$; $x:=1^c 0^{n-c}$;  \tcc*{initialize current level~$c$ and current vertex $x$} \label{line:init-xc}
$\nu_c:=1$; \lFor(\tcc*[f]{initialize $(\nu_1,\ldots,\nu_c)$ and $(p_1,\ldots,p_c)$}){$i:=1$ \KwTo $c$} {$p_i:=c-i+1$} \label{line:init-pnu}
\While{not enough vertices visited}{
 \leIf{$(c \text{ is even} \wedge x_1=0) \vee (c \text{ is odd} \wedge p_c<n \wedge x_{p_c+1}=0)$}{$\upw:=\true$}{$\upw:=\false$} \label{line:up-def}
 \If(\tcc*[f]{follow $\Gamma_n$ up}){$(\upw=\true \;\wedge\;\; c<l-1)$ \label{line:middle-up}}{
  \If(\tcc*[f]{$x_1=0$}){$c$ is even \label{line:case1-start}}{
   $x_1:=1$; $\Visit(x)$; $c:=c+1$; $p_c:=1$ \label{line:visit1} \;
   \leIf{$x_2=0$}{$\nu_c:=c$}{$\nu_c:=\nu_{c-1}$}
  }
  \Else(\tcc*[f]{$c$ is odd and $x_{i+1}=0$}){
   $i:=p_c$; $x_{i+1}:=1$; $\Visit(x)$; $c:=c+1$; $p_c:=i$; $p_{c-1}:=i+1$ \label{line:visit2} \;
   \leIf{$(i+1=n \;\vee\; x_{i+2}=0)$}{$\nu_c:=c-1$}{$\nu_c:=\nu_{c-2}$} \label{line:case1-end}
  }
 }
 \ElseIf(\tcc*[f]{follow $\Gamma_n$ down}){$(\upw=\false \;\wedge\;\; c>k+1)$ \label{line:middle-down}}{
  \If(\tcc*[f]{$x_1=1$}){$c$ is even \label{line:case2-start}}{
   $x_1:=0$; $\Visit(x)$; $c:=c-1$ \label{line:visit3} \;
   \lIf{$x_2=1$}{$\nu_c:=\nu_{c+1}$}
  }
  \Else(\tcc*[f]{$c$ is odd and $x_{i+1}=1$}){
   $i:=p_c$; $x_{i+1}:=0$; $\Visit(x)$; $c:=c-1$; $p_c:=i$; $\nu_c:=c$ \label{line:visit4} \;
   \lIf{$x_{i+2}=1$}{$\nu_{c-1}:=\nu_{c+1}$} \label{line:case2-end}
  }
 }
 \ElseIf(\tcc*[f]{follow $\Gamma_{n,c}$ through $\up(x,s_{\Gamma_{n,c}}(x))$}){$(\upw=\true \;\wedge\;\; c=l-1)$ \label{line:boundary-up}}{
  \If(\tcc*[f]{$x_{i-1}=0$ and $x_i=1$}){$c$ is even \label{line:case3-start}}{
   $i:=p_c$; $x_{i-1}:=1$; $\Visit(x)$; $x_i:=0$; $\Visit(x)$; $p_c:=i-1$ \label{line:visit5} \;
   \lIf{$x_{i+1}=1$}{$\nu_{c-1}:=\nu_c$; $\nu_c:=c$}
  }
  \Else(\tcc*[f]{$c$ is odd, $x_{i+1}=0$ and $x_i=1$}){
   $i:=p_c$; $x_{i+1}:=1$; $\Visit(x)$; $x_i:=0$; $\Visit(x)$; $p_c:=i+1$ \label{line:visit6} \;
   \lIf{$(i+2\leq n \;\wedge\; x_{i+2}=1)$}{$\nu_c:=\nu_{c-1}$} \label{line:case3-end}
  }
 }
 \Else(\tcc*[f]{$\upw=\false \;\wedge\;\; c=k+1$; follow $\Gamma_{n,c}$ through $\down(x,s_{\Gamma_{n,c}}(x))$} \label{line:boundary-down}){
  \leIf(\tcc*[f]{$i$ is minimal s.t.\ $x_i=0$}){$x_1=0$}{$i:=1$}{$i:=p_{\nu_c}+1$} \label{line:case4-start}
  \If(\tcc*[f]{$x_{i-2}=1$ and $x_i=0$}){$c \not\equiv i \bmod 2$}{
   $x_{i-2}:=0$; $\Visit(x)$; $x_i:=1$; $\Visit(x)$; $p_{{\nu_c}+1}:=i-1$; $p_{\nu_c}:=i$ \label{line:visit7} \;
   \leIf{$(i+1\leq n \;\wedge\; x_{i+1}=1)$}{$\nu_{{\nu_c}+1}:=\nu_{{\nu_c}-1}$}{$\nu_{{\nu_c}+1}:=\nu_c$}
   \lIf{$i>3$}{$\nu_c:=\nu_c+2$}
  }
  \Else(\tcc*[f]{$c \equiv i \bmod 2$; $j>i$ is minimal s.t.\ $x_j=1$}){
   \leIf{$x_1=0$}{$a:=c$; $j:=p_a$}{$a:={\nu_c}-1$; $j:=p_a$}
   \lIf{$j=n$}{$x_n:=0$; $\Visit(x)$; $x_i:=1$; $\Visit(x)$; $p_1:=i$; $\nu_c:=1$ \label{line:visit8}}
   \ElseIf(\tcc*[f]{$j<n$, $x_{i-1}=1$ and $x_{j+1}=0$}){$x_{j+1}=0$}{
    $x_{i-1}:=0$; $\Visit(x)$; $x_{j+1}:=1$; $\Visit(x)$; $p_{\nu_c}:=j$; $p_{{\nu_c}-1}:=j+1$ \label{line:visit9} \;
    \leIf{$(j+2\leq n \;\wedge\; x_{j+2}=1)$}{$\nu_{\nu_c}:=\nu_{{\nu_c}-2}$}{$\nu_{\nu_c}:={\nu_c}-1$}
    \lIf{$i>2$}{$\nu_c:=\nu_c+1$}
   }
   \Else(\tcc*[f]{$j<n$, $x_{j+1}=1$ and $x_i=0$}){
    $x_{j+1}:=0$; $\Visit(x)$; $x_i:=1$; $\Visit(x)$; $p_a:=i$; $p_{a-1}:=j$ \label{line:visit10} \;
    \lIf{$(j+2\leq n \;\wedge\; x_{j+2}=1)$}{$\nu_{a-2}:=\nu_a$}
    \leIf{$j=i+1$}{$\nu_c:=a-1$}{$\nu_{a-1}:=a-1$, $\nu_c:=a$} \label{line:case4-end}
   }
  }
 }
}
\end{algorithm}

Consider now the algorithm $\TrimGC(n,k,l)$ that computes a trimmed Gray code between levels $k$ and $l$ with $l-k\geq 2$ of $Q_n$ as described in Section~\ref{sec:trimming} in Theorem~\ref{thm:trim}, i.e., the algorithm produces a saturating cycle for cases~(i) and (ii) of Theorem~\ref{thm:sat}.
At the end of this section we describe how to modify the algorithm to produce a tight enumeration for cases~(i) and (ii) of Theorem~\ref{thm:enum}.

The algorithm maintains the current vertex in the variable $x$ and the current level in the variable $c$, $k<c<l$.
Both are initialized in line~\ref{line:init-xc} to $c=k+1$ and $x=1^c 0^{n-c}$, but the code can easily be modified to start at a different vertex.
The algorithm visits the sequence of vertices along the trimmed Gray code by the calls $\Visit(x)$, which could be some user-defined function that reads $x$.
For simplicity the main while-loop does not have an explicit termination criterion, but this can be added easily, e.g., by providing an additional argument to the algorithm that specifies the number of vertices to be visited before termination.
If the while-loop is not terminated, then the same cycle is traversed over and over again.
There are four main cases distinguished in the while-loop of the algorithm:
If the current level is between $k+1$ and $l-1$, then we either follow $\Gamma_n$ by flipping a 0-bit to a 1-bit if the condition in line~\ref{line:middle-up} is satisfied, or we follow $\Gamma_n$ by flipping a 1-bit to a 0-bit if the condition in line~\ref{line:middle-down} is satisfied.
If the current level is $c=l-1$ and $x$ is an upward vertex, i.e., the condition in line~\ref{line:boundary-up} is satisfied, then we first flip a 0-bit to the vertex $\up(x,y)$, $y:=s_{\Gamma_{n,c}}(x)$, and then a 1-bit to the vertex $y$.
On the other hand, if the current level is $c=k+1$ and $x$ is a downward vertex, i.e., the condition in line~\ref{line:boundary-down} is satisfied, then we first flip a 1-bit to the vertex $\down(x,y)$, $y:=s_{\Gamma_{n,c}}(x)$, and then a 0-bit to the vertex $y$.

The key to our loopless algorithm is to be able to determine in constant time the smallest integer $i\geq 1$ with $x_i=1$ or with $x_i=0$, recall \eqref{eq:GCsuc} and Lemmas~\ref{lem:GCupw}--\ref{lem:downexp}.
Furthermore, in Lemma~\ref{lem:downexp} we also need the smallest integer $j>i$ with $x_j=1$.
To achieve this the algorithm maintains the following data structures:
We maintain an array $(p_1,p_2,\dots,p_c)$ with the positions of the $1$-bits in $x=(x_1,x_2,\ldots,x_n)$ where $p_i$, $1\leq i \leq c$, is the position of the $i$-th 1-bit in $x$ counted from the right, from the highest index $n$.
Thus $p_c$ is the position of the first 1-bit in $x$ from the left, from the lowest index 1.
Since adding and removing $1$'s happens around the position $p_c$ (recall \eqref{eq:GCsuc}), the length of this array changes dynamically.
For $i>c$ the value of $p_i$ is undefined since the algorithm does not `clean up' those values after using them.
We also maintain an array $(\nu_1,\nu_2,\ldots,\nu_c)$ to quickly find the smallest integer $j>i$ with $x_j=0$ where $i$ is the smallest integer with $x_i=1$.
However, the value of $\nu_i$, $1\leq i \leq c$, is defined only if $p_i$ is a starting position of a substring of 1s in $x$; i.e., $x_{p_i}=1$ and $x_{p_i-1}=0$, or $p_i=1$.
In this case, the value of $\nu_i$ is the index such that $p_{\nu_i}$ is the position where the corresponding substring of 1s ends in $x$.
In particular, since $p_c$ is the position of the first $1$ in $x$, the value $p_{\nu_c}+1$ is the smallest integer greater than $p_c$ such that $x$ has a 0-bit at this position.
Initially, there is only one substring of 1s in $x=1^c 0^{n-c}$ that starts at position $p_c=1$ and ends at position $p_1=c$, so we have $\nu_c=1$, see line~\ref{line:init-pnu}.

\subsubsection{Correctness of the algorithm}

We now argue about the correctness of the algorithm $\TrimGC(n,k,l)$.
In lines~\ref{line:middle-up} and \ref{line:middle-down} (see also line~\ref{line:up-def}), by Lemma~\ref{lem:GCupw} the algorithm correctly checks whether $x$ is an upward vertex in a level $c\in[k+1,l-2]$ or a downward vertex in a level $c\in[k+2,l-1]$, respectively.
The correctness of the corresponding modifications in lines~\ref{line:case1-start}--\ref{line:case1-end} and lines~\ref{line:case2-start}--\ref{line:case2-end} follows immediately from \eqref{eq:GCsuc}.
In lines \ref{line:case3-start}--\ref{line:case3-end} the algorithm moves from $x$ in level~$c=l-1$ to $s_{\Gamma_{n,c}}(x)=:y$ via the common neighbor $\up(x,y)$ in level~$l$, which is correct by Lemma~\ref{lem:upexp}.
Similarly, in lines \ref{line:case4-start}--\ref{line:case4-end} the algorithm moves from $x$ in level~$c=k+1$ to $s_{\Gamma_{n,c}}(x)=:y$ via the common neighbor $\down(x,y)$ in level~$k$, which is correct by Lemma~\ref{lem:downexp}.
It can be verified straightforwardly from the pseudocode that the algorithm correctly updates all relevant entries of $p$ and $\nu$ in all cases.

\subsubsection{Running time and space requirements of the algorithm}

We first argue about the running time of the algorithm $\TrimGC(n,k,l)$.
Clearly, the initialization in lines~\ref{line:init-xc} and \ref{line:init-pnu} takes $\cO(n)$ time.
In all subsequent steps, there are only a constant number of operations between any two $\Visit(x)$ calls, so the algorithm is indeed loopless and each bitstring is generated in $\cO(1)$ time.
The space required by our algorithm is $\cO(n)$, as each of the arrays $x$, $p$ and $\nu$ has at most $n$ entries.

\subsubsection{Proof of part~(a) of Theorem~\ref{thm:algo}}

The previous arguments show that the algorithm $\TrimGC(n,k,l)$ correctly produces a saturating cycle in $Q_{n,[k,l]}$ in the claimed running time for cases~(i) and (ii) of Theorem~\ref{thm:sat}.
To obtain a loopless algorithm that generates a tight enumeration of the vertices of $Q_{n,[k,l]}$, we simply call $\TrimGC(n,k-1,l+1)$ and omit the first of the two $\Visit(x)$ calls in each of the lines~\ref{line:visit5}, \ref{line:visit6}, \ref{line:visit7}, \ref{line:visit8}, \ref{line:visit9} and \ref{line:visit10}.
The algorithm then moves directly with a distance-2 step from a vertex $x$ in level $k$ or $l$ to the vertex $s_{\Gamma_{n,k}}(x)$ or $s_{\Gamma_{n,l}}(x)$, respectively, without visiting $\down(x,s_{\Gamma_{n,k}}(x))$ or $\up(x,s_{\Gamma_{n,l}}(x))$ in between.
Also, if $k=0$ then line~\ref{line:init-pnu} has to be omitted, and the special case $c=1$ and $x=0^{n-1}1$ mentioned after \eqref{eq:GCsuc} must be treated separately in lines~\ref{line:visit4} and \ref{line:case2-end}.
This proves part~(a) of Theorem~\ref{thm:algo}.

\subsection{Efficient algorithms for glued Gray codes}

In this section we present an algorithm $\SatCycle(n,k)$ that computes a saturating cycle in $Q_{n,[k,k+1]}$, i.e., an algorithmic version of Theorem~\ref{thm:sat2}. The existence of such a cycle was first proved in \cite{MR3759914}.
This algorithm forms the basis for all further algorithms to produce saturating cycles and tight enumerations across an even number of levels of the cube (recall the gluing technique from the proofs of Theorem~\ref{thm:sat}~(iii) and Theorem~\ref{thm:enum}~(iiia) and (iiib)).
For space reasons we do not explicitly specify these more general algorithms as pseudocode in our paper, but rather explain how they can be derived from the algorithm $\SatCycle(n,k)$.
All algorithms are implemented in the C++ code we provide with this paper, to be found on our website \cite{www}, so implementation details can be looked up there.

\begin{algorithm}
\renewcommand\theAlgoLine{S\arabic{AlgoLine}}
\LinesNumbered
\DontPrintSemicolon
\SetEndCharOfAlgoLine{}
\SetNlSty{}{}{}
\SetArgSty{}
\SetKw{KwIf}{if}
\SetKw{KwElse}{else}
\SetKw{KwElseIf}{else if}
\SetKw{KwThen}{then}
\caption[Algorithm SatCycle()]{\mbox{$\SatCycle(n,k)$}}
\label{alg:satcycle2}
\vspace{.2em}
\KwIn{Integers $n\geq 3$ and $1\leq k\leq \lfloor(n-1)/2\rfloor$}
\KwResult{A saturating cycle in $Q_{n,[k,k+1]}$}
\vspace{.2em}
$x:=a_{n,k}$ \tcc*{initialize starting vertex} \label{line:init-x}
stack $S$ \tcc*{initialize empty stack $S$} \label{line:init-S}
$S.\push((\rec,(n,k,\rightarrow)))$ \tcc*{(3) visit $b_{n,k} \rightarrow a_{n,k}$ without $a_{n,k+1}$} \label{line:push-rec1}
$S.\push((\flip,n))$ \tcc*{(2) $a_{n,k+1} \rightarrow b_{n,k}$} \label{line:push-flip1a}
$S.\push((\flip,n-k))$ \tcc*{(1) $a_{n,k} \rightarrow a_{n,k+1}$} \label{line:push-flip1b}
\While {$S.\sempty()=\false$} {
  $(\instr,r):=S.\pop()$ \tcc*{pop from the stack} \label{line:pop}
  \If (\tcc*[f]{popped a $\flip$-item, perform one bitflip}) {$\instr=\flip$} {
    $i:=r$ \label{line:pop-flip} \;
    $x_i:=1-x_i$ \;
    $\Visit(x)$ \label{line:visit-x} \;
  }
  \Else (\tcc*[f]{popped a $\rec$-item, perform one recursion step}) {
    $(n,k,\dir):=r$ \label{line:local-nk} \;
    \If (\tcc*[f]{middle levels conjecture cycle}) {$k\geq 1 \wedge n=2k+1$ \label{line:rec-case1}} {
      \If {$\dir=\!{}\leftarrow$} {
        $\HamCycle(k,x,\leftarrow)$ \tcc*{visit $a_{n,k} \rightarrow b_{n,k}$ without $a_{n,k+1}$} \label{line:HCleft-call}
      }
      \Else {
        $\HamCycle(k,x,\rightarrow)$ \tcc*{visit $b_{n,k} \rightarrow a_{n,k}$ without $a_{n,k+1}$} \label{line:HCright-call}
      }
    }
    \ElseIf {$k\geq 1 \wedge n>2k+1$ \label{line:rec-case2}} {
      \If (\tcc*[f]{visit $a_{n,k} \rightarrow b_{n,k}$ without $a_{n,k+1}$}) {$\dir=\!{}\leftarrow$ \label{line:rec-case3}} {
        $S.\push((\rec,(n-1,k,\rightarrow)))$ \tcc*{(4) $b_{n-1,k}\circ 0\rightarrow a_{n-1,k}\circ 0=b_{n,k}$} \label{line:push-rec2a}
        $S.\push((\flip,n))$ \tcc*{(3) $b_{n-1,k}\circ 1 \rightarrow b_{n-1,k}\circ 0$} \label{line:push-flip2a}
        $S.\push((\flip,n-k-1))$ \tcc*{(2) $b_{n-1,k-1}\circ 1 \rightarrow b_{n-1,k}\circ 1$} \label{line:push-flip2b}
        $S.\push((\rec,(n-1,k-1,\leftarrow)))$ \tcc*{(1) $a_{n,k}=a_{n-1,k-1}\circ 1 \rightarrow b_{n-1,k-1}\circ 1$} \label{line:push-rec2b}
      }
      \Else (\tcc*[f]{visit $b_{n,k} \rightarrow a_{n,k}$ without $a_{n,k+1}$}) {
        $S.\push((\rec,(n-1,k-1,\rightarrow)))$ \tcc*{(4) $b_{n-1,k-1}\circ 1 \rightarrow a_{n-1,k-1}\circ 1=a_{n,k}$} \label{line:push-rec3a}
        $S.\push((\flip,n-k-1))$ \tcc*{(3) $b_{n-1,k}\circ 1 \rightarrow b_{n-1,k-1}\circ 1$} \label{line:push-flip3a}
        $S.\push((\flip,n))$ \tcc*{(2) $b_{n-1,k}\circ 0 \rightarrow b_{n-1,k}\circ 1$} \label{line:push-flip3b}
        $S.\push((\rec,(n-1,k,\leftarrow)))$ \tcc*{(1) $b_{n,k}=a_{n-1,k}\circ 0 \rightarrow b_{n-1,k}\circ 0$} \label{line:push-rec3b}
      }
    }
  }
}
\end{algorithm}

The algorithm $\SatCycle(n,k)$ is essentially a recursive implementation of the inductive proof of Theorem~\ref{thm:sat2}, based on induction over $n$, given in \cite{MR3759914}.
It starts at the vertex $x:=a_{n,k}$ (line~\ref{line:init-x}, recall the definition \eqref{eq:ai-bi}) and it uses a stack $S$, initialized in line~\ref{line:init-S}, to keep track of the recursion steps of the computation and thereby avoiding any recursive calls.
There are two different types of items pushed and popped from the stack, which we call $\flip$-item and $\rec$-items.
A $\flip$-item is a pair $(\flip,i)$, where the variable $i$, $1\leq i\leq n$, indicates the bit position in the current vertex $x$ to be flipped in some step.
Such items are pushed onto the stack $S$ in lines~\ref{line:push-flip1a}, \ref{line:push-flip1b}, \ref{line:push-flip2a}, \ref{line:push-flip2b}, \ref{line:push-flip3a} and \ref{line:push-flip3b}.
A $\rec$-item is a pair $(\rec,(n',k',\dir))$, $\dir\in\{\leftarrow,\rightarrow\}$, where the variables $n'$ and $k'$, $3\leq n'\leq n$, $1\leq k'\leq \lfloor(n'-1)/2\rfloor$, indicate the recursion step to compute a saturating path in $Q_{n',[k',k'+1]}$, and the variable $\dir$ indicates the direction in which this path is to be traversed.
By a \emph{saturating path} we mean a path that visits all vertices on level~$k'$, and that starts and ends on this level.
Specifically, if $\dir=\!{}\leftarrow$ then the saturating path in $Q_{n',[k',k'+1]}$ will be traversed starting at the vertex $a_{n',k'}$ and ending at the vertex $b_{n',k'}$, and it will omit the vertices $a_{n',k'+1}$ and $b_{n',k'+1}$. So, to complete this path to a saturating cycle, only the vertex $a_{n',k'+1}$ has to be added.
If $\dir=\!{}\rightarrow$ then the start and end vertices of the path are interchanged, the path starts at $b_{n',k'}$ and ends at $a_{n',k'}$.
Such $\rec$-items are pushed onto the stack $S$ in lines~\ref{line:push-rec1}, \ref{line:push-rec2a}, \ref{line:push-rec2b}, \ref{line:push-rec3a} and \ref{line:push-rec3b}.
Note that the values of $n$ and $k$ of the algorithm $\SatCycle(n,k)$ are overwritten in line~\ref{line:local-nk} with the values $n'$ and $k'$ of $\rec$-items previously pushed onto the stack.

In each iteration of the main while-loop of the algorithm, one item is popped from the stack (line~\ref{line:pop}).
For every $\flip$-item popped from the stack, the corresponding bit is flipped in $x$ (lines~\ref{line:pop-flip}--\ref{line:visit-x}).
For every $\rec$-item $(\rec,(n,k,\dir))$, $\dir\in\{\leftarrow,\rightarrow\}$, popped from the stack, there are three cases to consider:

If $k=0$, then we do nothing.
This is captured by a non-existent else-branch of the if-statement in line~\ref{line:rec-case1}.

If $k\geq 1$ and $n>2k+1$, then the $\rec$-item is replaced by two $\rec$-items and two $\flip$-items (lines~\ref{line:rec-case3}--\ref{line:push-rec3b}), i.e., we perform a recursion step to construct a saturating path in $Q_{n,[k,k+1]}$ by gluing together two saturating paths in $Q_{n-1,[k,k+1]}$ and $Q_{n-1,[k-1,k]}$. In the special case $k=1$, the latter path is degenerate.
The gluing that joins the two saturating paths to one is achieved by two bitflips enforced by the two $\flip$-items that are pushed onto the stack.
The exact details how this gluing is achieved will become clear in the next section.
Note however that in order to achieve a certain order of bitflips, the corresponding items have to be pushed onto the stack in reverse order.
The later execution order is indicated on the right-hand side of our pseudocode in the comments, indicated by numbers (1)--(3) and (1)--(4) in lines \ref{line:push-rec1}--\ref{line:push-flip1b}, \ref{line:push-rec2a}--\ref{line:push-rec2b} and \ref{line:push-rec3a}--\ref{line:push-rec3b}.

In the third case, if $k\geq 1$ and $n=2k+1$, then we encounter a nontrivial boundary case of our recursion, i.e., the middle levels conjecture, Theorem~\ref{thm:mlc}.
This case is handled by calling the constant-time algorithm $\HamCycle()$ for computing a Hamilton cycle in the graph $Q_{2k+1,[k,k+1]}$ presented in \cite{DBLP:conf/esa/MutzeN15, DBLP:conf/soda/MutzeN17}.
We modify this algorithm by applying suitable bit permutations so that it starts at the vertex $a_{2k+1,k}$ and ends at the vertex $b_{2k+1,k}$ or vice-versa, and so that it visits all vertices of $Q_{2k+1,[k,k+1]}$ except the first vertex $a_{2k+1,k}$ or $b_{2k+1,k}$, respectively, and except the vertex $a_{2k+1,k+1}$. To make this saturating path a Hamilton cycle, only this vertex has to be added.
We refer to the corresponding algorithm that starts at $a_{2k+1,k}$ and ends at $b_{2k+1,k}$ as $\HamCycle(k,x,\leftarrow)$ and to the algorithm that starts at $b_{2k+1,k}$ and ends at $a_{2k+1,k}$ as $\HamCycle(k,x,\rightarrow)$.
In these calls, the bitstring $x$ is an additional argument to which all internally computed bitflips by the algorithm $\HamCycle()$ are applied, followed by corresponding $\Visit(x)$ calls.
All this happens inside $\HamCycle()$.
Possibly $2k+1$ is strictly smaller than the length of $x$, but within the two $\HamCycle(k,x,\dir)$ calls, $\dir\in\{\leftarrow,\rightarrow\}$, in line~\ref{line:HCleft-call} and \ref{line:HCright-call}, only the first $2k+1$ bits of $x$ are modified.

\subsubsection{Correctness of the algorithm}

The next lemma states that the algorithm $\SatCycle(n,k)$ correctly glues together several saturating cycles of smaller dimension in the graphs $Q_{n',[k',k'+1]}$, where $3\leq n'\leq n$ and $1\leq k'\leq \lfloor(n'-1)/2\rfloor$, to a single saturating cycle in $Q_{n,[k,k+1]}$, visiting all the corresponding vertices along the way.

To state the lemma, we write $B_{n,k}$ for all bitstrings of length~$n$ with weight~$k$.
For $x=(x_1,x_2,\ldots,x_n)$ and $1\leq i,j\leq n$ we write $x_{[i,j]}:=(x_i,x_{i+1},\ldots,x_j)$. In the degenerate case $i>j$ we have $x_{[i,j]}=()$.
Moreover, we say that an item on the stack $S$ of our algorithm \emph{has been processed} at the moment the item below it is popped from the stack.
The item at the bottom of the stack $S$ \emph{has been processed} when the algorithm terminates.

\begin{lemma}
\label{lem:satcycle-correct}
The algorithm $\SatCycle(n,k)$ satisfies the following invariant for all inputs $n\geq 3$ and $1\leq k\leq \lfloor(n-1)/2\rfloor$:
\begin{enumerate}[label=(\roman*)]
\item
At the moment a $\rec$-item $(\rec,(n',k',\leftarrow))$ is popped from the stack $S$, the current vertex $x$ satisfies $x_{[1,n']}=a_{n',k'}$, and when this item has been processed, the current vertex $x$ satisfies $x_{[1,n']}=b_{n',k'}$.
Moreover, in between all vertices of the form $x'\circ x_{[n'+1,n]}$, $x'\in B_{n',k'}$, except the vertex $a_{n',k'}\circ x_{[n'+1,n]}$ are visited, no vertex is visited twice.
Furthermore, the vertex $a_{n',k'+1}\circ x_{[n'+1,n]}$ is not visited, and if $n'>2k'+1$, then the vertex $b_{n',k'+1}\circ x_{[n'+1,n]}$ is not visited either.
Put differently, $a_{n',k'}\circ x_{[n'+1,n]}$ together with the visited vertices forms a saturating path in $Q_{n',[k',k'+1]}\circ x_{[n'+1,n]}$.

\item
At the moment a $\rec$-item $(\rec,(n',k',\rightarrow))$ is popped from the stack $S$, the current vertex $x$ satisfies $x_{[1,n']}=b_{n',k'}$, and when this item has been processed, the current vertex $x$ satisfies $x_{[1,n']}=a_{n',k'}$.
Moreover, in between all vertices of the form $x'\circ x_{[n'+1,n]}$, $x'\in B_{n',k'}$, except the vertex $b_{n',k'} \circ x_{[n'+1,n]}$ are visited, no vertex is visited twice.
Furthermore, the vertex $b_{n',k'+1}\circ x_{[n'+1,n]}$ is not visited, and if $n'>2k'+1$, then the vertex $b_{n',k'+1} \circ x_{[n'+1,n]}$ is not visited either.
Put differently, $b_{n',k'} \circ x_{[n'+1,n]}$ together with the visited vertices forms a saturating path in $Q_{n',[k',k'+1]}\circ x_{[n'+1,n]}$.
\end{enumerate}
\end{lemma}

\begin{proof}
To prove the lemma we consider the recursion tree $T=T(n,k)$ corresponding to the algorithm $\SatCycle(n,k)$.
The tree $T$ is defined as follows, see Figure~\ref{fig:tnk}:
The tree is rooted and it is a binary tree, i.e., each node has either a left and a right child, or no children at all if it is a leaf.
The nodes of $T$ are the $\rec$-items that are stored on the stack $S$ in the course of the algorithm, and the $\rec$-item $(\rec,(n,k,\rightarrow))$ that is pushed onto the stack $S$ in the first step (line~\ref{line:push-rec1}) is the root.
The remaining nodes of $T$ are defined recursively:
Each node $(\rec,(n',k',\dir))$, $\dir\in\{\leftarrow,\rightarrow\}$, satisfying $k'=0$, or $k'\geq 1$ and $n'=2k'+1$ is a leaf.
Otherwise, if $\dir=\!{}\leftarrow$, then the node has the left child $(\rec,(n'-1,k'-1,\leftarrow))$ and the right child $(\rec,(n'-1,k',\rightarrow))$, and if $\dir=\!{}\rightarrow$, then the node has the left child $(\rec,(n'-1,k',\leftarrow))$ and the right child $(\rec,(n'-1,k'-1,\rightarrow))$.
By the instructions in lines~\ref{line:push-rec2a}--\ref{line:push-rec2b} and lines~\ref{line:push-rec3a}--\ref{line:push-rec3b}, left children correspond to $\rec$-items located at the top of the stack $S$ at the end of the iteration of the while-loop in which they are pushed onto $S$, whereas right children are $\rec$-items located three items below the top at this moment.
Note that several nodes of $T$ may have the same signature $(\rec,(n',k',\dir))$, see Figure~\ref{fig:tnk}, but these are different $\rec$-items encountered on the stack $S$ in the course of the algorithm.

To prove the lemma we prove the following auxiliary statements:
\begin{enumerate}[label=(\arabic*)]
\item
At the moment the item $(\rec,(n,k,\rightarrow))$ is popped from the stack, the current vertex $x$ satisfies $x=b_{n,k}$.

\item
Given a non-leaf node $(\rec,(n',k',\leftarrow))$ and its left child $(\rec,(n'-1,k'-1,\leftarrow))$ in $T$, suppose that at the moment the item $(\rec,(n',k',\leftarrow))$ is popped from the stack the current vertex $x$ satisfies $x_{[1,n']}=a_{n',k'}$.
Then at the moment the left child $(\rec,(n'-1,k'-1,\leftarrow))$ is popped from the stack the current vertex $x$ satisfies $x_{[1,n'-1]}=a_{n'-1,k'-1}$.

\item
Given a non-leaf node $(\rec,(n',k',\rightarrow))$ and its left child $(\rec,(n'-1,k',\leftarrow))$ in $T$, suppose that at the moment the item $(\rec,(n',k',\rightarrow))$ is popped from the stack the current vertex $x$ satisfies $x_{[1,n']}=b_{n',k'}$.
Then at the moment the left child $(\rec,(n'-1,k,\leftarrow))$ is popped from the stack the current vertex $x$ satisfies $x_{[1,n'-1]}=a_{n'-1,k'}$.

\item
Given a non-leaf node $(\rec,(n',k',\leftarrow))$ and its right child $(\rec,(n'-1,k',\rightarrow))$ in $T$, suppose that at the moment the item $(\rec,(n'-1,k'-1,\leftarrow))$ (this is the left child) has been processed the current vertex $x$ satisfies $x_{[1,n'-1]}=b_{n'-1,k'-1}$.
Then at the moment the right child $(\rec,(n'-1,k',\rightarrow))$ is popped from the stack the current vertex $x$ satisfies $x_{[1,n'-1]}=b_{n'-1,k'}$.

\item
Given a non-leaf node $(\rec,(n',k',\rightarrow))$ and its right child $(\rec,(n'-1,k'-1,\rightarrow))$ in $T$, suppose that at the moment the item $(\rec,(n'-1,k',\leftarrow))$ (this is the left child) has been processed the current vertex $x$ satisfies $x_{[1,n'-1]}=b_{n'-1,k'}$.
Then at the moment the right child $(\rec,(n'-1,k'-1,\rightarrow))$ is popped from the stack the current vertex $x$ satisfies $x_{[1,n'-1]}=b_{n'-1,k'-1}$.

\item
Given a node $(\rec,(n',k',\dir))$, $\dir\in\{\leftarrow,\rightarrow\}$, in $T$, suppose that at the moment that this item is popped from the stack the current vertex $x$ satisfies $x_{[1,n']}=a_{n',k'}$ if $\dir=\!{}\leftarrow$ or $x_{[1,n']}=b_{n',k'}$ if $\dir=\!{}\rightarrow$.
Moreover, if this node is not a leaf, then suppose both children satisfy properties~(i) or (ii) of the lemma.
Then, this item satisfies property~(i) or (ii) of the lemma (if $\dir=\!{}\leftarrow$ or $\dir=\!{}\rightarrow$, respectively).
\end{enumerate}

Note that all these auxiliary statements except (1) are conditional.
However, the unconditional properties~(i) and (ii) of the lemma can be easily derived from the conditional statements (1)--(6) by applying them to all nodes of the recursion tree $T$ along a left-to-right tree traversal, starting by applying (1) to the root of $T$, which is the item $(\rec,(n,k,\rightarrow))$, see the dashed line in Figure~\ref{fig:tnk}.

\begin{figure}
\centering
\includegraphics[scale=0.916]{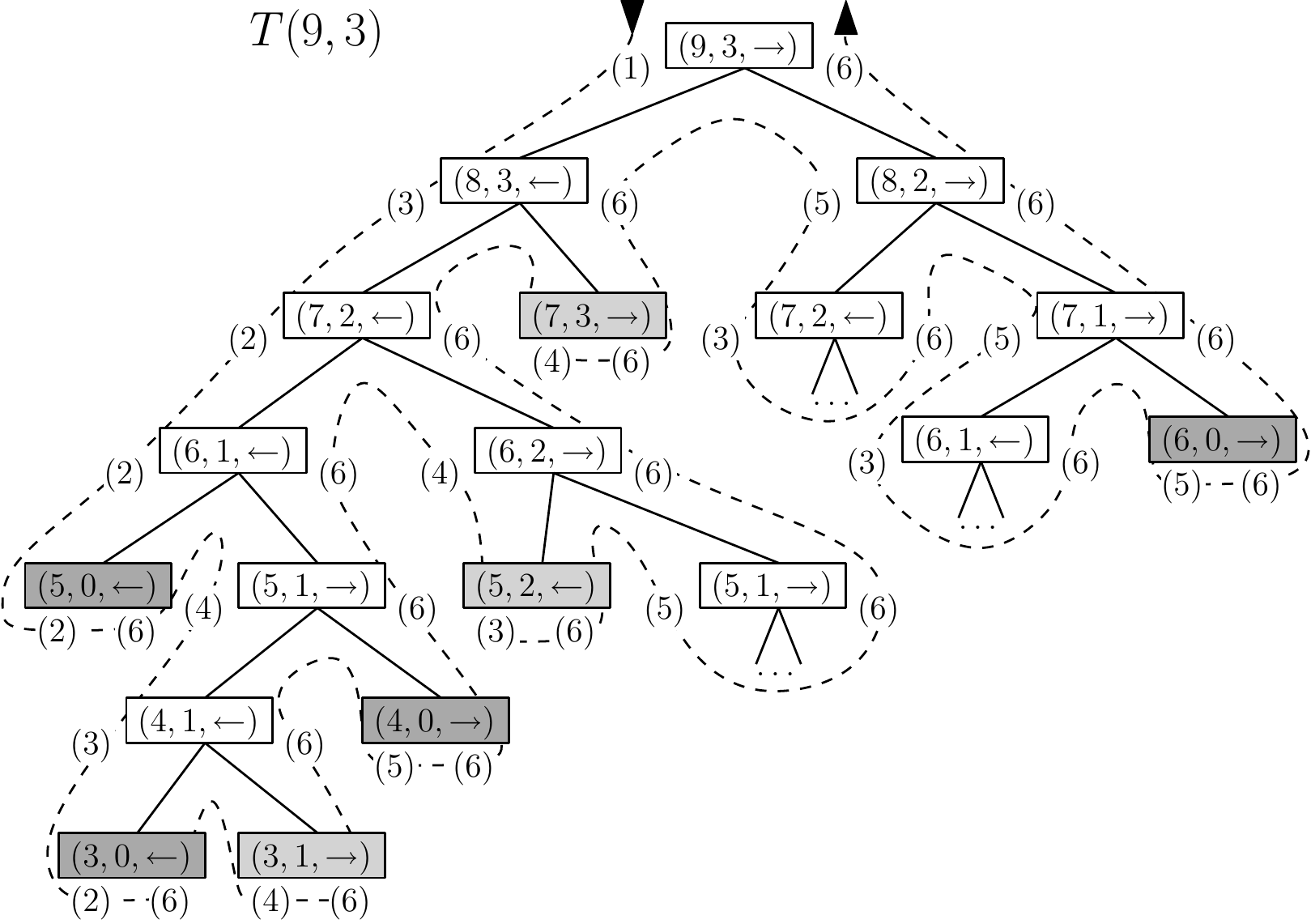} 
\caption{The recursion tree $T=T(n,k)$ for $n=9$ and $k=3$ and other notations used in the proof of Lemma~\ref{lem:satcycle-correct}.
In the figure, the nodes $(\rec,(n',k',\dir))$ of $T$ are represented by the triples $(n',k',\dir)$.
The dashed line indicates a left-to-right tree traversal order in which the conditional auxiliary statements (1)--(6) are applied to derive the unconditional properties~(i) and (ii) of the lemma.}
\label{fig:tnk}
\end{figure}

We proceed to prove the auxiliary statements mentioned before.

\begin{enumerate}[label=(\arabic*)]
\item
This follows easily from the instructions in lines~\ref{line:init-x}--\ref{line:push-flip1b}.
Specifically, the item $(\rec,(n,k,\rightarrow))$ is popped from the stack in the third iteration of the while-loop, after the bits at positions $n-k$ and $n$ have been flipped, leading from the initial vertex $a_{n,k}$ via $a_{n,k+1}$ to the vertex $x=b_{n,k}$.

\item
When the node $(\rec,(n',k',\leftarrow))$ is popped from the stack, it is replaced by four new items, and the left child $(\rec,(n'-1,k'-1,\leftarrow))$ becomes the new topmost item on the stack, see lines~\ref{line:push-rec2a}--\ref{line:push-rec2b}.
Therefore, this item will be popped from the stack in the next iteration of the while-loop, i.e., the value of $x$ is not modified in between.
From the assumption $x_{[1,n']}=a_{n',k'}$ and the definition \eqref{eq:ai-bi} we obtain that $x_{[1,n'-1]}=a_{n'-1,k'-1}$.

\item
When the node $(\rec,(n',k',\rightarrow))$ is popped from the stack, it is replaced by four new items, and the left child $(\rec,(n'-1,k',\leftarrow))$ becomes the new topmost item on the stack, see lines~\ref{line:push-rec3a}--\ref{line:push-rec3b}.
Therefore, this item will be popped from the stack in the next iteration of the while-loop, i.e., the value of $x$ is not modified in between.
From the assumption $x_{[1,n']}=b_{n',k'}$ and the definition \eqref{eq:ai-bi} we obtain that $x_{[1,n'-1]}=a_{n'-1,k'}$.

\item
When the node $(\rec,(n',k',\leftarrow))$ is popped from the stack, it is replaced by four new items, where the left child $(\rec,(n'-1,k'-1,\leftarrow))$ becomes the new topmost item on the stack and the right child $(\rec,(n'-1,k',\rightarrow))$ is located three items below, see lines~\ref{line:push-rec2a}--\ref{line:push-rec2b}.
So when the left child has been processed, then the algorithm runs two iterations where the bits at positions $n'-k'-1$ (line~\ref{line:push-flip2b}) and $n'$ (line~\ref{line:push-flip2a}) are flipped, and in the next iteration the right child $(\rec,(n'-1,k',\rightarrow))$ is popped from the stack.
From the assumption that before these two bitflips we have $x_{[1,n'-1]}=b_{n'-1,k'-1}$ and from the definition \eqref{eq:ai-bi} we obtain that after these two bitflips the current vertex $x$ satisfies $x_{[1,n'-1]}=b_{n'-1,k'}$.

\item
When the node $(\rec,(n',k',\rightarrow))$ is popped from the stack, it is replaced by four new items, where the left child $(\rec,(n'-1,k',\leftarrow))$ becomes the new topmost item on the stack and the right child $(\rec,(n'-1,k'-1,\rightarrow))$ is located three items below, see lines~\ref{line:push-rec3a}--\ref{line:push-rec3b}.
So when the left child has been processed, then the algorithm runs two iterations where the bits at positions $n'$ (line~\ref{line:push-flip3b}) and $n'-k'-1$ (line~\ref{line:push-flip3a}) are flipped, and in the next iteration the right child $(\rec,(n'-1,k'-1,\rightarrow))$ is popped from the stack.
From the assumption that before these two bitflips we have $x_{[1,n'-1]}=b_{n'-1,k'}$ and from the definition \eqref{eq:ai-bi} we obtain that after these two bitflips the current vertex $x$ satisfies $x_{[1,n'-1]}=b_{n'-1,k'-1}$.

\item
We distinguish three cases. By the conditions in lines~\ref{line:rec-case1} and \ref{line:rec-case2}, no other cases are possible:
\begin{enumerate}[label=(\alph*)]
\item
The node $(\rec,(n',k',\dir))$ is a leaf and we have $k'=0$ (the corresponding leafs of $T$ are highlighted in dark gray in Figure~\ref{fig:tnk}).
We consider the subcase that $\dir=\!{}\leftarrow$, and we need to prove that this item satisfies property~(i) of the lemma.
In this case, the item is simply popped from the stack (the condition on the value of $x$ in this moment holds by assumption), and nothing is done, i.e., the item has been processed in the next iteration of the while-loop.
In between, the variable $x$ is not modified, so by the assumption $x_{[1,n']}=a_{n',0}=0^{n'}$ it also satisfies $x_{[1,n']}=b_{n',0}=0^{n'}$.
Trivially, in between all vertices of the form $x'\circ x_{[n'+1,n]}$, $x'\in B_{n',0}=\{0^{n'}\}$, except the vertex $a_{n',0}\circ x_{[n'+1,n]}=0^{n'}\circ x_{[n'+1,n]}$ (i.e., no vertices at all) are visited.
Furthermore, neither the vertex $a_{n',1}\circ x_{[n'+1,n]}$ nor the vertex $b_{n',1}\circ x_{[n'+1,n]}$ are visited.
This shows that the item $(\rec,(n',k',\dir))$ indeed satisfies property~(i) of the lemma.
The other subcase that $\dir=\!{}\rightarrow$ can be handled analogously.
In this case property~(ii) also holds trivially.

\item
The node $(\rec,(n',k',\dir))$ is a leaf and we have $k'\geq 1$ and $n'=2k'+1$. The corresponding leafs of $T$ are highlighted in lightgray in Figure~\ref{fig:tnk}.
We consider the subcase that $\dir=\!{}\leftarrow{}$, and we need to prove that this item satisfies property~(i) of the lemma.
In this case, the item is popped from the stack (the condition on the value of $x$ in this moment holds by assumption), and the algorithm issues the call $\HamCycle(k',x,\leftarrow)$ in line~\ref{line:HCleft-call}.
We know that upon termination of this algorithm, the current value of the variable $x$ satisfies $x_{[1,n']}=b_{n',k'}$.
Moreover, within the algorithm $\HamCycle()$ all vertices of the form $x'\circ x_{[n'+1,n]}$, $x'\in B_{n',k'}$, except the vertex $a_{n',k'}\circ x_{[n'+1,n]}$ are visited.
Furthermore, the vertex $a_{n',k'+1}\circ x_{[n'+1,n]}$ is not visited.
This shows that the item $(\rec,(n',k',\dir))$ indeed satisfies property~(i) of the lemma.
The other subcase that $\dir=\!{}\rightarrow$ can be handled analogously.
In this case property~(ii) holds by our definition of the algorithm $\HamCycle(k',x,\rightarrow)$ called in line~\ref{line:HCright-call}.

\item
The node $(\rec,(n',k',\dir))$ is not a leaf, so we have $k'\geq 1$ and $n'>2k'+1$ (the corresponding inner nodes of $T$ are drawn in white in Figure~\ref{fig:tnk}).
We consider the subcase that $\dir=\!{}\leftarrow{}$, and we need to prove that this item satisfies property~(i) of the lemma.
In this case, the item is popped from the stack (the condition on the value of $x$ in this moment holds by assumption) and replaced by four new items $(\rec,(n'-1,k',\rightarrow))$, $(\flip,n')$, $(\flip,n'-k'-1)$ and $(\rec,(n'-1,k'-1,\leftarrow))$, pushed onto the stack in this order in lines~\ref{line:push-rec2a}--\ref{line:push-rec2b}.
By the assumptions that the left child $(\rec,(n'-1,k'-1,\leftarrow))$ of the node $(\rec,(n',k',\dir))$ in $T$ satisfies property~(i) of the lemma and the right child $(\rec,(n'-1,k',\rightarrow))$ satisfies property~(ii) of the lemma, we obtain the following:
When the item $(\rec,(n',k',\dir))$ is popped from the stack and replaced by the four new items mentioned above, by assumption we know that the current vertex $x$ satisfies $x=a_{n',k'}\circ x_{[n'+1,n]}=a_{n'-1,k'-1}\circ 1\circ x_{[n'+1,n]}$.
Until the left child $(\rec,(n'-1,k'-1,\leftarrow))$ has been processed, the algorithm visits all vertices of the form $x'\circ 1\circ x_{[n'+1,n]}$, $x'\in B_{n'-1,k'-1}$, except the vertex $a_{n'-1,k'-1}\circ 1\circ x_{[n'+1,n]}=a_{n',k'}\circ x_{[n'+1,n]}$, no vertex twice.
Furthermore, neither the vertex $a_{n'-1,k'}\circ 1\circ x_{[n'+1,n]}=a_{n',k'+1}\circ x_{[n'+1,n]}$ nor the vertex $b_{n'-1,k'}\circ 1\circ x_{[n'+1,n]}$ are visited.
At this point, the current vertex $x$ satisfies $x=b_{n'-1,k'-1}\circ 1\circ x_{[n'+1,n]}$.
After that, the $\flip$-item $(\flip,n'-k'-1)$ is popped from the stack, so in the next iteration of the while-loop the algorithm visits the vertex $x=b_{n'-1,k'}\circ 1\circ x_{[n'+1,n]}$ for the first time.
After that, the $\flip$-item $(\flip,n')$ is popped from the stack, so in the next iteration of the while-loop the algorithm visits the vertex $x=b_{n'-1,k'}\circ 0\circ x_{[n'+1,n]}$ for the first time.
Until the right child $(\rec,(n'-1,k',\rightarrow))$ has been processed, the algorithm visits all vertices of the form $x'\circ 0\circ x_{[n'+1,n]}$, $x'\in B_{n'-1,k'}$, except the vertex $b_{n'-1,k'}\circ 0\circ x_{[n'+1,n]}$ (this vertex has been visited just before), no vertex twice.
Furthermore, neither the vertex $a_{n'-1,k'+1}\circ 0\circ x_{[n'+1,n]}=b_{n',k'+1}\circ x_{[n'+1,n]}$ nor the vertex $b_{n'-1,k'+1}\circ 0\circ x_{[n'+1,n]}$ are visited.
At this point, the current vertex $x$ satisfies $x=a_{n'-1,k'}\circ 0\circ x_{[n'+1,n]}=b_{n',k'}\circ x_{[n',n]}$.
Summarizing, the item $(\rec,(n',k',\dir))$ indeed satisfies property~(i) in the lemma.
The other subcase that $\dir=\!{}\rightarrow$ can be handled analogously, considering which four items are pushed onto the stack in lines~\ref{line:push-rec3a}--\ref{line:push-rec3b} after the item $(\rec,(n',k',\dir))$ is popped.
\end{enumerate}
\end{enumerate}

This completes the proof of the lemma.
\end{proof}

From Lemma~\ref{lem:satcycle-correct} we conclude that the algorithm $\SatCycle(n,k)$ indeed computes a saturating cycle in $Q_{n,[k,k+1]}$:
By the instructions in lines~\ref{line:init-x}--\ref{line:push-flip1b}, we start at the vertex $x=a_{n,k}$, and in the first two iterations of the while-loop, two $\flip$-items are popped from the stack in the reverse order in which they are pushed onto the stack, which entail the calls $\Visit(a_{n,k+1})$ and $\Visit(b_{n,k})$.
In the next iteration, the $\rec$-item $(\rec,(n,k,\rightarrow))$ is popped from the stack, so by property~(ii) from Lemma~\ref{lem:satcycle-correct}, from this point in time until this item has been processed, i.e., until the algorithm terminates, the algorithm visits a saturating path in $Q_{n,[k,k+1]}$ that ends at the vertex $a_{n,k}$. This is the last visited vertex.
Together with the first two visited vertices this saturating path forms a saturating cycle in the graph $Q_{n,[k,k+1]}$.

\subsubsection{Running time and space requirements of the algorithm}

The running time of the algorithms $\HamCycle(k,x,\dir)$, $\dir\in\{\leftarrow,\rightarrow\}$, called in lines~\ref{line:HCleft-call} and \ref{line:HCright-call} is $\cO(1)$ on average per visited vertex plus $\cO(k)$ for the initialization. As mentioned in \cite{DBLP:conf/soda/MutzeN17}, the initialization time is only $\cO(k)$ instead of the general bound $\cO(k^2)$ when the starting vertex is prescribed, as $a_{2k+1,k}$ or $b_{2k+1,k}$ in our case.
Since in total $2\binom{2k+1}{k}-2=2^{\Theta(k)}$ vertices are visited in each call, the $\cO(k)$ initialization time can be discounted to the $\cO(1)$ average time per visited vertex.

The next lemma is needed to bound the running time of the remaining operations of the algorithm $\SatCycle()$.

\begin{lemma}
\label{lem:stack}
The algorithm $\SatCycle(n,k)$ satisfies the following invariant for all inputs $n\geq 3$ and $1\leq k\leq \lfloor(n-1)/2\rfloor$:
\begin{enumerate}[label=(\roman*)]
\item
Any two consecutive $\rec$-items on the stack $S$ are separated by exactly two $\flip$-items.
At the bottom of the stack $S$ is a single $\rec$-item, and on the top of the stack are at at most two $\flip$-items.

\item
Consider all $\rec$-items $(\rec,(n_1,k_1,\dir_1)),\ldots (\rec,(n_r,k_r,\dir_r))$, on the stack $S$ from bottom to top.
Then we have $n_1>n_2>\cdots >n_{r-1}\geq n_r\geq 3$.
\end{enumerate}
\end{lemma}

\begin{proof}
These properties follow immediately from the instructions in lines~\ref{line:push-rec1}--\ref{line:push-flip1b}, \ref{line:local-nk}, \ref{line:push-rec2a}--\ref{line:push-rec2b}, \ref{line:push-rec3a}--\ref{line:push-rec3b}.
\end{proof}

From Lemma~\ref{lem:stack}~(ii) we conclude that the stack $S$ has height at most $3n$.

Consider the set $R$ of all $\rec$-items and the set $F$ of all $\flip$-items that appear on the stack $S$ in the course of the algorithm. Thus $R\cup F$ is the node set of the tree $T(n,k)$ defined in the previous section.
We define an injection $R\rightarrow F$ as follows, see Lemma~\ref{lem:stack}~(i): The $\rec$-item pushed onto the stack in line~\ref{line:push-rec1} is mapped onto the $\flip$-item pushed onto the stack in line~\ref{line:push-flip1a}.
Moreover, the $\rec$-items pushed in lines~\ref{line:push-rec2a} and \ref{line:push-rec2b} are mapped onto the $\flip$-items pushed in lines~\ref{line:push-flip2a} and \ref{line:push-flip2b}, respectively.
Similarly, the $\rec$-items pushed in lines~\ref{line:push-rec3a} and \ref{line:push-rec3b} are mapped onto the $\flip$-items pushed in lines~\ref{line:push-flip3a} and \ref{line:push-flip3b}, respectively.
From this injection it follows that $|R|\leq |F|$.
Note that processing each stack item takes only constant time, where by processing we mean popping it from the stack and either issuing the call to $\HamCycle()$, without the time spent inside this function, or pushing four other items onto the stack.
Combining this with the bound $|F|\geq |R|$ and using that each $\flip$-item leads to a call $\Visit(x)$ in line~\ref{line:visit-x} yields that the algorithm spends on average $\cO(1)$ time per visited vertex (outside of $\HamCycle()$).
As we argued before, also the algorithm $\HamCycle()$ needs only $\cO(1)$ time on average to visit each vertex, so overall each vertex is generated in $\cO(1)$ time on average.

The space required by our algorithm is only $\cO(n)$, as the array $x$ and the stack $S$ have at most $\cO(n)$ entries.

\subsubsection{Proof of part~(b) of Theorem~\ref{thm:algo}}

First of all the algorithm $\SatCycle(n,k)$ can easily be generalized to allow parameters $\lfloor n/2\rfloor\leq k\leq n-2$ by starting at the vertex $x:=1^{k+1}0^{n-k-1}$ and by applying the bitflip sequence computed by the original algorithm $\SatCycle(n,n-k-1)$, which starts at the complement of $x$, the vertex $a_{n,n-k-1}$.
Moreover, by mimicking the gluing approach from the proof of Theorem~\ref{thm:sat}~(iii), the algorithm $\SatCycle(n,k)$ can be easily generalized to an algorithm $\SatCycle(n,k,l)$ that generates a saturating cycle in $Q_{n,[k,l]}$ for any even number of consecutive levels $[k,l]$ that are entirely below or above the middle, by suitably stringing together multiple calls $\SatCycle(n,k')$, $k\leq k'\leq l$.
For details how this is done, see our C++ implementation.
All the previously established bounds for the running time and space requirements carry over to this setting.

To obtain analogous algorithms for tight enumerations, one first derives an algorithm $\TightEnum(n,k)$ that computes a tight enumeration of the vertices of $Q_{n,[k,k+1]}$. This algorithm mimicks the proof of Theorem~\ref{thm:enum}~(iiia) and is very much analogous to the algorithm $\SatCycle(n,k)$, using a stack.
Mimicking the gluing approach from the proof of Theorem~\ref{thm:enum}~(iiib), this algorithm can then be generalized to an algorithm $\TightEnum(n,k,l)$ that generates a tight enumeration of the vertices of $Q_{n,[k,l]}$ for any even number of levels $[k,l]$ that are entirely below or above the middle. For details, see our C++ implementation.

This proves part~(b) of Theorem~\ref{thm:algo}.

\section{Proof of Theorem~\ref{thm:long}}
\label{sec:long}

We now present the proof of Theorem~\ref{thm:long}.

\begin{proof}[Proof of Theorem~\ref{thm:long}]
For $c=0$ the claim follows from Theorem~\ref{thm:mlc}, so we can assume that $c\geq 1$.

The cycle obtained from Theorem~\ref{thm:trim} by trimming the reflected Gray code to levels $[k-c,k+1+c]$ of $Q_{2k+1}$ misses exactly
\begin{equation}
\label{eq:mb1}
  m:=2\Big(\binom{2k+1}{k+c+1}-u_{2k+1,k+c}\Big) \eqBy{eq:udnk} 2\left(\binom{2k+1}{k+c+1}-\binom{2k}{k+c}\right)=2\binom{2k}{k+c+1}
\end{equation}
many vertices.
The total number of vertices of $Q_{2k+1,[k-c,k+1+c]}$ satisfies the relation
\begin{equation}
\label{eq:vb}
  v:=2\left(\binom{2k+1}{k+c+1}+\binom{2k+1}{k+c}+\cdots+\binom{2k+1}{k+1}\right)\geq 2(c+1)\binom{2k+1}{k+c+1}\geq 4(c+1)\binom{2k}{k+c+1}
\end{equation}
(recall \eqref{eq:vQ}).
Combining \eqref{eq:mb1} and \eqref{eq:vb} shows that
\begin{equation}
\label{eq:mv1}
  \frac{m}{v}\leq \frac{1}{2(c+1)} \enspace,
\end{equation}
yielding the first bound claimed in Theorem~\ref{thm:long}.

To derive the second bound, we consider another way of building a long cycle in $Q_{2k+1,[k-c,k+1+c]}$.
For odd $c\geq 1$ we glue together $c+1$ many saturating cycles, obtained from Theorem~\ref{thm:sat2}, between the pairs of consecutive levels $(k-c,k-c+1),\ldots,(k-1,k),(k+1,k+2),\ldots,(k+c,k+c+1)$ as in the proof of Theorem~\ref{thm:sat}~(iii), see Figure~\ref{fig:gluing}.
The number of missed vertices $m$ in this case satisfies the relation
\begin{align}
  m &= 2\left(\binom{2k+1}{k+1}-\binom{2k+1}{k+2}+\binom{2k+1}{k+3}-\binom{2k+1}{k+4}+\cdots+\binom{2k+1}{k+c}-\binom{2k+1}{k+c+1}\right) \notag \\
    &= 2\left(\binom{2k}{k}-\binom{2k}{k+c+1}\right) \leq 2 \Big(\Big(\frac{k+1}{k-c}\Big)^{c+1}-1\Big)\binom{2k}{k+c+1} \enspace. \label{eq:mb2}
\end{align}
Combining \eqref{eq:vb} and \eqref{eq:mb2} and using that $1-x\leq \exp(-x)$ we obtain
\begin{equation}
\label{eq:mv2}
  \frac{m}{v}\leq \frac{1}{2(c+1)}\Big(\Big(\frac{k+1}{k-c}\Big)^{c+1}-1\Big)
  \leq \frac{1}{2(c+1)} \Big(\exp\Big(\frac{(c+1)^2}{k-c}\Big)-1\Big) \enspace.
\end{equation}
For even $c\geq 2$ we glue together $c+1$ many saturating cycles between the pairs of consecutive levels $(k-c,k-c+1),\ldots,(k,k+1),\ldots,(k+c,k+c+1)$ as before.
In this case the number of missed vertices $m$ can be computed very similarly to before as $2(\binom{2k}{k+1}-\binom{2k}{k+c+1})$, and this expression is bounded from above by the expressions in \eqref{eq:mb2}, implying that the same bound \eqref{eq:mv2} holds also in this case.

Combining the bounds \eqref{eq:mv1} and \eqref{eq:mv2} completes the proof of the theorem.
\end{proof}

\section*{Acknowledgements}

The authors thank Ji\v{r}\'{i} Fink, Jerri Nummenpalo and Robert \v{S}\'{a}mal for several stimulating discussions about the problems discussed in this paper. We also thank the reviewers for their helpful comments and suggestions.
The first author acknowledges the support by the Czech Science Foundation grant GA14-10799S.

\bibliographystyle{alpha}
\bibliography{../refs}

\end{document}